\documentclass[11pt,letterpaper]{article}
\usepackage[margin=1in]{geometry}
\usepackage{amssymb}
\usepackage{amsmath,ulem}
\usepackage{hyperref}
\usepackage{setspace}
\usepackage{enumitem}
\setlist{nolistsep}
\newcommand{\PG}{\textup{PG}}

\newcommand{\Q}{\mathcal{Q}}
\newcommand{\B}{\mathcal{B}}
\newcommand{\T}{\mathcal{T}}
\renewcommand{\H}{\mathcal{H}}
\renewcommand{\P}{\mathcal{P}}
\renewcommand{\L}{\mathcal{L}}
\newcommand{\K}{\mathcal{K}}

\newcommand{\C}{\mathcal{C}}
\renewcommand{\S}{\mathcal{S}}

\newtheorem{theorem}{Theorem}[section]

\newtheorem{lemma}[theorem]{Lemma}
\newtheorem{corollary}[theorem]{Corollary}

\newtheorem{remark}[theorem]{Remark}
\newtheorem{proposition}[theorem]{Proposition}

\newtheorem{example}{Example}

\newenvironment{proof}{\noindent{\bfseries Proof}\hspace{0.5em}}{ \null \hfill $\square$ \par}
\usepackage[usenames,dvipsnames]{color}

\begin{document}
\title{Quasi-polar spaces}
\author{Jeroen Schillewaert\thanks{JS is supported by a University of Auckland Faculty Development Research Fund} \and Geertrui Van de Voorde\thanks{This author is supported by the Marsden Fund Council administered by the Royal Society of New Zealand.}
\date{}
}
\maketitle
\begin{abstract}
Quasi-polar spaces are sets of points having the same intersection numbers with respect to hyperplanes as classical polar spaces. Non-classical examples of quasi-quadrics have been constructed using a technique called {\em pivoting} \cite{QQ}. We introduce a more general notion of pivoting, called switching, and also extend this notion to Hermitian polar spaces. 

The main result of this paper studies the switching technique in detail by showing that, for $q\geq 4$, if we modify the points of a hyperplane of a polar space to create a quasi-polar space, the only thing that can be done is pivoting. The cases $q=2$ and $q=3$ play a special role for parabolic quadrics and are investigated in detail. Furthermore, we give a construction for quasi-polar spaces obtained from pivoting multiple times.

Finally, we focus on the case of parabolic quadrics in even characteristic and determine under which hypotheses the existence of a nucleus (which was included in the definition given in \cite{QQ}) is guaranteed. 
\end{abstract}

\section{Introduction}

The characterisation of polar spaces by their intersection properties can be traced back to the 1950’s, with seminal work done by Segre and Tallini and their schools. In general, polar spaces are not characterised by their intersection numbers with hyperplanes. In \cite{QQ}, the authors constructed sets which were not quadrics but shared the intersection numbers with quadrics. We refer to \cite{QD} for a survey of these classical results.

We define a {\em quasi-polar space} to be a set of points in $\PG(m,q)$, $m\geq 2$, with the same intersection sizes with hyperplanes as a non-degenerate classical polar space embedded in $\PG(m,q)$.

Consequently, in $\PG(2n+1,q)$ we will distinguish between {\em elliptic} and {\em hyperbolic} quasi-quadrics, while in $\PG(2n,q)$ the only quasi-quadrics are {\em parabolic}. Quasi-Hermitian varieties exist in odd and even dimension. We also say that the quasi-polar space is of elliptic, hyperbolic, parabolic or Hermitian {\em type}. Throughout this paper, all polar spaces will be assumed to be non-symplectic.

We will call quadrics of types $\Q^+(2n+1,q)$ and $\Q^-(2n+1,q)$ of {\em opposite type}. When talking about hyperplanes meeting in an elliptic (or hyperbolic) quadric, we say that the {\em hyperplane} is {\em of elliptic (or hyperbolic) type}. A hyperplane meeting a quasi-quadric in the number of points of a singular quadric is a {\em singular hyperplane}; the singular hyperplanes of a quadric are then precisely those meeting in a cone with vertex a point and base a quadric of the same type. We denote a cone with vertex $P$ and base $\Q$ as $P\Q$.

In this paper, we introduce the notion of {\em switching}, which takes as input a quasi-polar space $\P$ in $\PG(m,q)$. One fixes a hyperplane $\pi$ of $\PG(m,q)$ in which a set $R$ of points is replaced with a set $R'$ forming the new set $\P' = (\P \setminus R) \cup R'$. We say that $\P'$ is obtained by {\em switching $\P$ in the hyperplane $\pi$}. {\em Pivoting} is a particular type of switching, and was introduced in \cite{QQ}: we {\em pivot} a polar space $\P$ if we switch in a singular hyperplane $\pi$ of $\P$, and replace the cone $\pi\cap \P=P\Q$ by $P\Q'$ where $\Q'$ is a quasi-polar space of the same type as $\Q$.

{\bf Outline and main results of the paper:}  We refer to the statements in the paper for a precise version of the statements in this overview, since many have exceptions for small $q$ and are technical to state. In Section \ref{sec:cardinality} we show that in all cases except for the parabolic quasi-quadric, the number of points of a quasi-polar space easily follows from the definition.
In Section \ref{sec:switching} we discuss switching, and we determine under which conditions a set $\P'$ obtained by switching a quasi-polar space is a quasi-polar space. In Subsection \ref{sec:type-preserving} we show that
switching preserves the type of the quasi-polar space (Lemma \ref{lem:typepreserved}). In Subsection \ref{sec:impossible} we show that switching in non-singular hyperplanes is not possible 
(Lemma \ref{lem:switching-nonsingular}). In Section \ref{sec:switching-singular} we investigate switching in singular hyperplanes.
Our main result appears in Subsection \ref{sec:switching-pivoting}. It states that if $\mathcal{P}$ is not a parabolic quadric with $q$ even, then switching in a singular hyperplane with vertex $P$ is pivoting (Theorem \ref{prop:switching-is-pivoting}). 
In Section \ref{sec:repeated} we provide a repeated pivoting construction (Proposition \ref{repeatedpivoting}). In Section \ref{PQQ} we investigate parabolic quasi-quadrics when $q$ is even. In Subsection \ref{sec:FDC} we study the original definition given in \cite{QQ} where a parabolic quasi-quadric is required to have a nucleus. We study the latter's properties in Subsection \ref{sec:nucleus-properties}. In Subsection \ref{sec:pivoting-PQQ} we show that we can pivot in a parabolic quadric and more generally obtain a quasi-quadric without nucleus (Proposition \ref{nonucleus}). We also determine what happens if we require the quasi-parabolic quadric to retain a nucleus (Corollary \ref{cor:parabolic}).
We conclude the paper in Subsection \ref{sec:nucleus-sufficient} by providing a sufficient condition for a parabolic quasi-quadric to have a nucleus (Lemma \ref{thereisnucleus}), along with a set of seemingly weaker conditions which prove to be equivalent to those of a parabolic quasi-quadric with nucleus (Proposition \ref{equivalentdef}).
Finally, some of the proofs of Propositions in this paper which are similar to others have been collected in Appendix A, and results for $\Q(2n,2)$ (Proposition \ref{Q2n2}) and $\Q(2n,3)$ (Proposition \ref{Q2n3}) appear in Appendix B.

\section{Cardinality of quasi-polar spaces} \label{sec:cardinality}

We first show that in all cases except for parabolic quasi-quadrics and two exceptional low-dimensional cases, the number of points of a quasi-polar space easily follows from the definition.
Throughout this paper, we will use the following convention to treat the elliptic an hyperbolic quadric simultaneously: we use the $\pm$ symbol where $\pm$ reads as $+$ when we are in the hyperbolic case and $-$ in the elliptic case (and vice versa for $\mp$). All statements and their proofs should be read by choosing either the top or bottom row of the symbols $\pm$ and $\mp$ consistently (and as such every statement containing the symbol $\pm$ proves two different statements, one for the hyperbolic case, and one for the elliptic case.)

\begin{lemma}\label{lem:correct-points}\begin{itemize}
\item[(i)] Let $\S$ be a set of points in $\PG(2n+1,q)$, $n\geq 1$, such that every hyperplane meets in $|\Q(2n,q)|$ or $|P\Q^\pm(2n-1,q)|$ points, then $|\S|=|\Q^\pm(2n+1,q)|$ unless $n=1$ and $\S$ is the set of $q+1$ points on a line in $\PG(3,q)$ (thus meeting every plane in $|\Q(2,q)|=q+1$ or $|P\Q^-(1,q)|=1$ points).

\item[(ii)] Let $\T$ be a set of points in $\PG(m,q^2)$, $m\geq 2$, such that every hyperplane meets $\T$ in $|\H(m-1,q^2)|$ or $|P\H(m-2,q^2)|$ points, then $|\T|=|\H(m,q^2)|$ unless $m=2$ and $\T$ is the set of $q^2+q+1$ points of a Baer subplane (thus meeting every line in $|\H(1,q^2)|=q+1$ or $|P\H(0,q^2)|=1$ points).

\end{itemize}
Moreover, in all of the above cases, except for $n=1$ in (i) and $m=2$ in (ii), the number of hyperplanes of a fixed type is a constant.
\end{lemma}
\begin{proof}\begin{itemize}\item[(i)] Let $\alpha$ be the number of hyperplanes meeting $\S$ in $u=|\Q(2n,q)|$ points and $\beta$ be the number of hyperplanes meeting $\S$ in $v=|P\Q^\pm(2n-1,q)|$ points. Standard counting yields that
\begin{align}
\alpha+\beta&=\frac{q^{2n+2}-1}{q-1}\\ 
\alpha u+\beta v&=|\S|\frac{q^{2n+1}-1}{q-1}\\ 
\alpha u (u-1)+\beta v (v-1)&=|\S|(|\S|-1)\frac{q^{2n}-1}{q-1}.
\end{align}

Using the first two equations to write $\alpha$ and $\beta$ in function of $|\S|$, we see that the third yields a quadratic equation in $|\S|$ whose sum of roots is given by 
\begin{align*}
\Sigma=1+\frac{(q^{2n+1}-1)(u+v-1)}{q^{2n}-1}.
\end{align*} 
One of those roots is $|\S|=|\Q^\pm(2n+1,q)|$, which is an integer so the other root is an integer if and only $\Sigma$ is an integer.
Now $u+v-1=2\frac{q^{2n}-1}{q-1}\pm q^n-1$. This implies that $\Sigma$ is an integer if and only if $q^{2n}-1$  divides
$$(q^{2n+1}-1)(2\frac{q^{2n}-1}{q-1}\pm q^n-1),$$
so if and only if $q^{2n}-1|(q^{2n+1}-1)(\pm q^n-1)$. This is equivalent to $(q^{2n}-1)|(q^n\mp1)(q^{2n+1}-1)$, and hence, to the condition that $q^n\pm 1$ divides $q^{2n+1}-1$.

We see that in the hyperbolic case, this never happens, while in the elliptic case, $\Sigma$ is an integer if and only if $n=1$. In that case, $q+1$ is the root for $\S$, different from $|\Q^-(3,q)|$. Since $\S$ is a set of $q+1$ points in $\PG(3,q)$ such that every plane intersects it in $1$ or $q+1$ points, it is easy to see (see also \cite{BB}) that in this case, the $q+1$ points of $\S$ form a line.

\item[(ii)] We proceed in the same way as above. Let $\alpha$ be the number of hyperplanes meeting $\T$ in $r=|\H(m-1,q^2)|$ points and $\beta$ be the number of hyperplanes meeting $\T$ in $s=|P\H(m-2,q^2)|$ points. 

Standard counting now yields that
\begin{align}
\alpha+\beta&=\frac{q^{2m+2}-1}{q^2-1}\\
\alpha r+\beta s&=|\T|\frac{q^{2m}-1}{q^2-1}\\
\alpha r(r-1)+\beta s(s-1)&=|\T|(|\T|-1)\frac{q^{2m-2}-1}{q^2-1}.
\end{align}

Using the first two equations to write $\alpha$ and $\beta$ in function of $|\T|$, we see that the third yields a quadratic equation in $|\T|$ whose sum of roots is given by 
\begin{align*}
\Sigma=1+\frac{(q^{2m}-1)(r+s-1)}{q^{2m-2}-1}.
\end{align*} 
One of those roots is $|\T|=|\H(m,q^2)|$, which is an integer so the other root is an integer if and only $\Sigma$ is an integer. 
Using $\pm$ and $\mp$ where the top row represents the case that $m$ is even and the bottom row represents the case that $m$ is odd, we find that
 $r+s-1=\frac{q^{m-1}\pm1}{q^2-1}(2q^m\mp(q^2+1))$. This implies that $\Sigma$ is an integer if and only if $q^{2m-2}-1$  divides
$$\frac{(q^{2m}-1)(q^{m-1}\pm1)(2q^m\mp(q^2+1))}{q^2-1}$$
so if and only if $q^{m-1}\mp1|\frac{q^{2m}-1}{q^2-1}(2q^m\mp(q^2+1))$. This is equivalent to $(q^{m-1}\mp1)|\frac{(q^{m}-1)(q^m+1)(q-1)^2}{q^2-1}$, and hence simplifies further to the condition that $q^{m-1}\mp1$ divides $(q^m\pm1)(q-1)$, and hence, $q^{m-1}\mp1 | (q-1)^2$.
It follows that there are no solutions if $m$ is odd, and for $m$ is even, we only find a solution when $m=2$. In the latter case, we find that $q^2+q+1$ is the second root for $|\T|$, different from $|\H(2,q^2)|$, and it is easy to see that in this case, $\T$ is the set of points of a Baer subplane.

\end{itemize}

Finally, we see that the values of $\alpha,\beta$ are uniquely determined when $|\S|$, resp. $|\T|$ is fixed.
\end{proof}

\section{Switching in quasi-polar spaces}\label{sec:switching}

We investigate switching in more detail. First we show that switching is type-preserving (Subsection \ref{sec:type-preserving}), and then we show that switching is impossible in non-singular hyperplanes (Subsection \ref{sec:impossible}).

\subsection{Switching is type-preserving}\label{sec:type-preserving}

The following lemma shows that, if $q\notin\{ 2,4\}$, switching preserves the type of a quasi-polar space. 

\begin{lemma}\label{lem:typepreserved} Let $\P$ be a quasi-polar space. Suppose that $\P'$ is a quasi-polar space obtained by switching in the hyperplane $\pi$. Then $\P'$ is a quasi-polar of the same type as $\P$, unless $\{\P,\P'\}$ is the set of
\begin{itemize}
\item a parabolic quasi-quadric in $\PG(2,4)$ and a Baer subplane $\PG(2,2)$ in $\PG(2,4)$, or
\item an elliptic and hyperbolic quasi-quadric in $\PG(2n+1,2)$.
\end{itemize}
\end{lemma}
\begin{proof} 
Suppose first that $\P$ is a quasi-Hermitian variety in $\PG(2n+1,q^2)$, $n\geq 1$. Then $\P$ meets $\pi$ in either $C_1=|\H(2n,q^2)|$ or $C_2=|P\H(2n-1,q^2)|$ points. Note that $C_2>C_1$ for all $q$.
Now suppose that $\P'$ is a quasi-polar space of a different type, then necessarily, $\P'$ is an elliptic or hyperbolic quasi-quadric. Hence, $|\P'\cap\pi|=D_1=|P\Q^\pm(2n-1,q^2)|$, or $|\P'\cap\pi|=D_2=|Q(2n,q^2)|$. 
Suppose first that $\P'$ is not the set of $q^2+1$ points on a line.
Then we know from Lemma \ref{lem:correct-points} that $|\P'|=|\Q^\pm(2n+1,q^2)|=\frac{(q^{2n+2}\mp 1)(q^{2n} \pm 1)}{q^2-1}$ and $|\P|=|\H(2n+1,q^2)|=(q^{2n+1}+1)\frac{q^{2n+2}-1}{q^2-1}$. Thus $|\H(2n+1,q^2)|-|\Q^\pm(2n+1,q^2)| =\frac{q^{2n}(q^{2n+2}-1)}{q+1}+q^{2n}\mp q^{2n}$.

Since  $\frac{q^{2n}(q^{2n+2}-1)}{q+1}+q^{2n}\mp q^{2n}>C_2>C_1$ for all $n\geq 1,q\geq 2$, we see that for all $i,j$, $|\H(2n+1,q^2)|\neq |\Q^\pm(2n+1,q^2)|+C_i-D_j$. Recall that $\P$ and $\P'$ coincide outside $\pi$. Since the number of points of $\P$, not in $\pi$ is $|\H(2n+1,q^2)|-C_i$ for some $i$, and the number of points of $\P'$, not in $\pi$, is $|\Q^\pm(2n+1,q^2)|-D_j$ for some $j$, this is a contradiction.

If $n=1$ and $|\P'|=q^2+1$, we find that $C_2=q^3+q^2+1$, $D_1=1$ and $D_2=q^2+1$. As above, we see that $|\H(3,q^2)|-(q^2+1)>C_2$, so $|\H(3,q^2)|-(q^2+1)\neq C_i-D_j$ for all $i,j$, a contradiction. By reversing the roles of $\P$ and $\P'$, we deduce that it is impossible for a switched quasi-quadric in $\PG(2n+1,q)$ to be a quasi-Hermitian variety.

Suppose that $\P$ is a quasi-quadric embedded in $\PG(2n+1,q)$ such that $\P'$ is a quasi-quadric of a different type. W.l.o.g. suppose that $\P$ is a hyperbolic quasi-quadric and $\P'$ is an elliptic quasi-quadric. By Lemma \ref{lem:correct-points}, if $|\P'|\neq |\Q^-(2n+1,q)|$, then $\P'$ is the set of $q+1$ points on a line. It is impossible for $\P'$ to be the set of $q+1$ points on a line of $\PG(3,q)$ since the complement of a hyperplane meets $\Q^+(3,q)$ in at least $q^2$ points; whereas it has $q$ or $0$ points in common with the $q+1$ points of a line.
This implies that $|\P'|=|\Q^-(2n+1,q^2)|$ and hence, that the hyperplane $\pi$ needs to contain $|\Q^+(2n+1,q)|-|\Q^-(2n+1,q)|=2q^{n}$ points of $\P\setminus(\P\cap\P')$; let   $\S=\P\setminus(\P\cap\P')$, then $\S$ is a set of $2q^n$ points contained in $\pi$. 

Consider a hyperplane $\mu$, different from (possibly) $\pi$, meeting $\P$ in $|P\Q^+(2n-1,q)|$ points. Since this hyperplane contains $|\Q(2n,q)|$ or $|P\Q^-(2n-1,q)|$ points of $\P'$, it has to intersect $\S$ in  $|P\Q^+(2n-1,q)|-|\Q(2n,q)|=q^n$ or $|P\Q^+(2n-1,q)|-|P\Q^-(2n-1,q)|=2q^n$  points, which are contained in the codimension $2$-space $\pi\cap \mu$. 

Moreover, a hyperplane that has $|\Q(2n,q)|$ points of $\P$ either contains $|\Q(2n,q)|$ or $|P\Q^-(2n-1,q)|$ points of $\P'$. This implies that every hyperplane of $\pi$ meets $\S$ in $0,q^n$ or $2q^{n}$ points. Standard counting yields that
\begin{align}
\alpha+\beta+\gamma&=\frac{q^{2n+1}-1}{q-1}\\ 
\beta q^n+\gamma 2q^n&=2q^n\frac{q^{2n}-1}{q-1}\\ 
\beta q^n (q^n-1)+\gamma 2q^n (2q^n-1)&=2q^n(2q^n-1)\frac{q^{2n-1}-1}{q-1},
\end{align}
where $\alpha, \beta,\gamma$ denotes the number of hyperplanes of $\pi$ meeting $\S$ in $0,q^n,2q^n$ points respectively. Subtracting the second from the last equation and simplifying gives that 
$$\gamma=\frac{-q^{3n}+q^{2n}+2q^{3n-1}-q^{2n-1}-q^n}{q^n(q-1)},$$ so $\gamma<0$ if $q> 2$, a contradiction. 

If $q=2$, we find that $(\alpha,\beta,\gamma)=(2^{n-1},2^{2n+1}-2^n,2^{n-1}-1)$ and $\S$ is a set of $2^{n+1}$ points in $\PG(2n,2)$ meeting every hyperplane in $0,2^n$ or $2^{n+1}$ points. We will show that $\S$ is the set of points of an $(n+1)$-dimensional affine subspace. Since $\gamma$ equals the number of hyperplanes containing an $(n+1)$-dimensional space in $\PG(2n,2)$ and since $|\S| = 2^{n+1}$ the intersection of hyperplanes containing $\S$ is an $(n+1)$-dimensional space 
$\Pi_{n+1}$. Since $\alpha>0$, there is a hyperplane $H$ with $0$ points of $\S$. This hyperplane necessarily meets $\Pi_{n+1}$ in a hyperplane $\mu$ of $\Pi_{n+1}$, so all points of $\S$ are contained in the affine subspace $\Pi_{n+1}\setminus \mu$. Since $\S$ has $2^{n+1}$ points, $\S$ equals this affine subspace.

Consider an $(n-1)$-space contained in a generator of $\Q^+(2n+1,2)$, then we see that the symmetric difference of the two $n$-spaces thus found forms an affine $(n+1)$-space. It follows that it is possible to switch $\Q^+(2n+1,q)$ to obtain a quasi-elliptic quadric.

Finally, suppose that $\P$ is a quasi-Hermitian variety in $\PG(2n,q^2)$. Let $\pi$ be a hyperplane of $\PG(2n,q^2)$. Suppose that $\P'$ is a parabolic quasi-quadric obtained from switching $\P$ in $\pi$.
First assume that $n>1$. Note that $\pi$ meets $\P$ in at most $|\H(2n-1,q^2)|$ points since $|\H(2n-1,q^2)|>|P\H(2n-2,q^2)|$.
 Since there are $\frac{q^{4n}-1}{q^2-1}$ hyperplanes of $\pi$ and $\frac{q^{4n-2}-1}{q^2-1}$ hyperplanes of $\pi$ through a point of $\pi$, this implies that there exists a hyperplane $\mu$ of $\pi$ with at most $k:=|\H(2n-1,q^2)|\frac{q^{4n-2}-1}{q^{4n}-1}$ points of $\pi$.
A hyperplane $\rho$ of $\PG(2n,q^2)$ through $\mu$ has at least $|P\H(2n-2,q^2)|$ points of $\P$, so it has at least $|P\H(2n-2,q^2)|-k$ points of $\P$ outside of $\pi$. This implies that $|\rho\cap \P'|\geq |P\H(2n-2,q^2)|-k$.

Now $\rho\cap\P'$ is either $|\Q^+(2n-1,q^2)|,|\Q^-(2n-1,q^2)|$ or $|P\Q(2n-2,q^2)|$, and hence, $\rho\cap \P'$ has at most $|\Q^+(2n-1,q^2)|$ points. Since $|P\H(2n-2,q^2)|-k>|\Q^+(2n-1,q^2)|$ for $n>1$ and $q\geq 2$, we obtain a contradiction.

Now let $n=1$. From Lemma \ref{lem:correct-points}, we know that a quasi-Hermitian variety in $\PG(2,q^2)$ has either $q^3+1$ points or is a Baer subplane $\PG(2,q^2)$.

First assume that $\P$ has $q^3+1$ points. Then there are at least $q^3-q$ points of $\P$ outside $\pi$. Let $R$ be a point of $\pi$, not in $\P$, then $R$ lies on $q^2$ lines, different from $\pi$. Since each of these lines has either $1$ or $q+1$ points of $\P$, $R$ lies on at least one line with $q+1$ points of $\P$. Since these $q+1$ points belong to $\P'$, and $q+1>2$, we find a contradiction.

Now assume that $\P$ is a Baer subplane. Since $\P'$ is a parabolic quasi-quadric, every line meets $\P'$ in $0,1$ or $2$ points. Any line, different from $\pi$ meeting $\P$ in $q+1$ points has at least $q$ points of $\P'$, a contradiction if $q>2$. If $(n,q)=(1,2)$, then a Baer subplane $\B=\PG(2,2)$ in $\PG(2,4)$ is a quasi-Hermitian variety. Let $L$ be a line meeting $\B$ in $3$ points. If we switch $\P$ in $L$ by removing the $3$ points of $\B\cap L$, we find a set $\P'$ of $4$ points such that every line meets it in $0=|\Q^-(1,4)|$, $1=|P\Q(0,4)|$ or $2=|\Q^+(1,4)|$ points, which means that $\P'$ is a parabolic quasi-quadric.
By reversing the roles of $\P$ and $\P'$, we deduce that it is impossible for a switched quasi-quadric in $\PG(2n,q)$ where $(n,q) \neq (1,2)$ to be a quasi-Hermitian variety. 
\end{proof}

\begin{corollary}\label{samesize} Let $\P$ be an elliptic or hyperbolic quasi-quadric in $\PG(2n+1,q)$, $q\neq 2$, or a quasi-Hermitian variety in $\PG(m,q)$, and let $\P'$ be the quasi-polar space obtained by switching $\P$ with respect to a hyperplane $\pi$. Then $|\P|=|\P'|$.
\end{corollary}
\begin{proof} Let $\P$ be as in the statement, then we have shown in Lemma \ref{lem:typepreserved} that the type of $\P$ does not change when switching. If $\P$ is an elliptic or hyperbolic quasi-quadric or a quasi-Hermitian variety, then by Lemma \ref{lem:correct-points} the number of points in $\P$ is determined by its type, so $|\P|=|\P'|$, unless $\P$ is an elliptic quasi-quadric in $\PG(3,q)$ or $\P$ is a quasi-Hermitian variety in $\PG(2,q)$, $q$ square. If $\P$ is an elliptic quasi-quadric in $\PG(3,q)$ of size $q^2+1$ and $\P'$ is an elliptic quasi-quadric of size $q+1$, then the complement of $\pi$ contains at least $q^2-q$ points of $\P$, whereas the complement of $\pi$ contains at most $q$ points of $\P'$. Since $\P$ and $\P'$ coincide outside $\pi$, and $q>2$ this is a contradiction. Similarly, if $\P$ is a quasi-Hermitian variety in $\PG(2,q)$ of size $q\sqrt{q}+1$ and $\P'$ is a Baer subplane, then the complement of $\pi$ contains at least $q\sqrt{q}-\sqrt{q}$ points of $\P$ and at most $q+\sqrt{q}$ points of $\P'$, a contradiction if $q\neq 4$. If $q=4$, we see that the $6$ points obtained by removing $3$ points of a secant line to a unital cannot be the set of $6$ points obtained by removing one point of a Baer subplane; for the latter set there is a unique point lying on three $2$-secants, whereas the former set has at most two $2$-secants through each point.
\end{proof}

\subsection{Switching in non-singular hyperplanes is impossible}\label{sec:impossible}

We start with an easy lemma which we will use frequently.

\begin{lemma}\label{lem:equal-pointsets}
Let $\mathcal{P}$ and $\mathcal{P'}$ be two point sets in $\PG(m,q)$, such that for every hyperplane $\pi$ of $\PG(m,q)$ one has $|\pi\cap\P|=|\pi\cap\P'|$. Then $\mathcal{P}=\mathcal{P}'$.
\end{lemma}
\begin{proof}
Counting pairs $(x,H)$, where $x\in \mathcal{P}$ and where $H$ is a hyperplane containing $x$, in two ways yields that
$|\mathcal{P}|\frac{q^m-1}{q-1} = \sum_{H\in \PG(m,q)} |H \cap \mathcal{P}|$. Similarly, we have that $|\mathcal{P'}|\frac{q^m-1}{q-1} = \sum_{H\in \PG(m,q)} |H \cap \mathcal{P'}|$.
Since $|H \cap \mathcal{P}|=|H \cap \mathcal{P'}|$ for all $H$, it follows that $|\mathcal{P}|=|\mathcal{P'}|$.
 
 Let $Q$ be a point in $\PG(m,q)$ and let $I_\P(Q)$ be the indicator function with respect to $\mathcal{P}$. Considering triples $(Q,R,H)$ where the point $R\neq Q$ is contained in $\P$ and the hyperplane $H$ contains both $Q$ and $R$ we obtain $$(|\mathcal{P}|-I_\P(Q))\frac{q^{m-1}-1}{q-1} = \sum_{H\ni p} (|H\cap \mathcal{P}|  - I_\P(Q)).$$ A similar equation holds for $\mathcal{P'}$, with indicator function $I_{\P'}$. Since $|\mathcal{P}|=|\mathcal{P'}|$ and since $|H\cap \mathcal{P}| = |H\cap \mathcal{P'}|$ for all hyperplanes $H$ this implies that $I_\P(Q)=I_{\P'}(Q)$. Hence, for all points $Q$, we have that $Q\in \P$ if and only if $Q\in \P'$, so $\P=\P'$.
\end{proof}

In most of what follows, we will consider quasi-polar spaces in $\PG(m,q)$, $m\geq 3$. The following proposition shows that this condition is not a restriction since switching for quasi-polar spaces contained in $\PG(2,q)$ is essentially trivial. We use standard terminology to call a parabolic quasi-quadric of size $|\Q(2,q)|=q+1$ an {\em oval}, and a quasi-Hermitian variety in $\PG(2,q)$ a {\em unital}.

\begin{proposition} \label{prop:planecases} 
Let $\P$ be a unital in $\PG(2,q)$, $q>4$ a square, and let $\P'$ be a unital obtained by switching $\P$. Then $\P=\P'$.

Let $\P$ be an oval in $\PG(2,q)$, $q>3$ odd,  and let $\P'$ be an oval obtained by switching $\P$ in the line $L$. Then $\P=\P'$.

Let $\P$ be an oval in $\PG(2,q)$, $q$ even, and let $\P'$ be an oval obtained by switching $\P$ in the line $L$. Then either $\P=\P'$ or $L$ is a tangent to $\P$ in the point $P$ and $\P'$ consists of the points of $\P\setminus \{P\}\cup \{N\}$ where $N$ is the nucleus of $\P$.
\end{proposition}
\begin{proof} Let $\P$ be a unital. Recall that $|\P|=|\P'|$ by Corollary \ref{samesize}, and that $\P$ and $\P'$ coincide outside $\pi$, so $|\P\cap \pi|=|\P'\cap \pi|$. Let $Q$ be a point of $\pi$, not in $\P$, then $Q$ lies on at least one tangent line and at least one secant line to $\P$, different from $\pi$. It follows that $Q\notin \P'$ for all points on $\pi$ not in $\P$. But since $|\P\cap \pi|=|\P'\cap \pi|$, this implies that the points of $\pi\cap \P$ need to be contained in $\pi\cap \P'$. Hence, $\P=\P'$.

Let $\P$ be an oval in $\PG(2,q)$, $q$ odd. Since we assume that $\P'$ is an oval, we have that $|\P\cap L|=|\P'\cap L|$. If $Q$ is a point of $\pi$, not in $\P$, then, since $q>3$ is odd, $Q$ lies on at least one $2$-secant and at least one passant to $\Q$, different from $\pi$. It follows that $Q\notin \P'$, and hence, as above, that $\P=\P'$.

Finally, let $\P$ be an oval in $\PG(2,q)$, $q$ even and let $N$ be its nucleus.  Since we assume that $\P'$ is an oval, we have that $|\P\cap L|=|\P'\cap L|$, so if $\P\neq \P'$, necessarily $L$ is a secant or tangent line. Every point $Q$ of $L$ not in $\P$, different from the nucleus $N$, lies on at least one passant and $2$-secant. It follows that $Q\notin \P'$, and hence, that $\P'$ is different from $\P$ only if $L$ is a tangent line to $\P$, say in the point $P$, and $\P'$ is obtained by removing the point $P$ and adding the point $N$. 
\end{proof}

\begin{lemma}\label{lem:switching-nonsingular} Let $\P$ be a quadric or a Hermitian variety in $\PG(m,q)$, $m\geq 3$ with $q\neq 2$ for elliptic and hyperbolic quadrics and $q\geq 4$ for parabolic quadrics. 
Let $\P'\neq \P$ be a quasi-quadric or quasi-Hermitian variety obtained from switching $\P$ in the hyperplane $\pi$, then $\pi$ is a singular hyperplane.
\end{lemma}
\begin{proof} {\bf Case 1: $\P=\mathcal{Q}^\pm(2n+1,q)$}. We know that every hyperplane meets $\P$ in $|\Q(2n,q)|$ or $|P\Q^\pm(2n-1,q)|$ points. Suppose that $\pi$ is a non-singular hyperplane (i.e., meeting $\P$ in a $\Q(2n,q)$), and that $\P'$ is a quasi-quadric obtained by switching in $\pi$. We will deduce that $\P=\P'$ by showing that the conditions of Lemma \ref{lem:equal-pointsets} are met.

From Corollary \ref{samesize}, we know that  $\P'$ has size $|\P|$, which implies that $|\pi\cap \P'|=|\Q(2n,q)|$. Let $\mu$ be a hyperplane in $\pi$, then $\mu$ meets $\Q(2n,q)$ in either an elliptic quadric, a cone $P\Q(2n-2,q)$ or a hyperbolic quadric. 

If $\mu$ is a hyperplane of $\pi$ meeting $\P$ in a $\Q^\mp(2n-1,q)$, then $\mu$ lies only on non-singular hyperplanes, since a $P\Q^\pm(2n-1,q)$ does not have hyperplanes meeting in a $\Q^\mp(2n-1,q)$. That means that $\mu$ has either $|\Q^\mp(2n-1,q)|$ or $|\Q^\mp(2n-1,q)|$+ $||P\Q^\pm(2n-1,q)|-|\Q(2n,q)||=|\Q^\mp(2n-1,q)|+q^{n}$ points of $\P'$ in $\pi$.

If $\mu$ is a hyperplane of $\pi$ meeting $\P$ in a $\Q^\pm(2n-1,q)$, then $\mu$ lies on exacty two singular and $q-1$ non-singular hyperplanes, one of which is $\pi$. Hence, if $q>2$, there is at least one hyperplane through $\mu$, different from $\pi$ with $||\Q(2n,q)|-|\Q^\pm(2n-1,q)||$ points of $\P$ outside of $\mu$ and there are two hyperplanes through $\mu$ with $|P\Q^\pm(2n-1,q)|-|\Q^\pm(2n-1,q)|$ points outside of $\mu$. This means that $\mu$ has to meet $\P'$ in exactly $|\Q^\pm(2n-1,q)|$ points.

If $\mu$ is a hyperplane of $\pi$ meeting $\P$ in a $P\Q(2n-2,q)$, then $\mu$ lies on a unique hyperplane meeting $\P$ in a $P\Q^\pm(2n-1,q)$ and $q$ hyperplanes meeting $\P$ in a $Q(2n,q)$, one of which is $\pi$. It follows that $\mu$ needs to contain exactly $|P\Q(2n-2,q)|$ points of $\P'$.

So we find that $\T=\P'\cap \pi$ is a set of $|\Q(2n,q)|$ points in $\PG(2n,q)$ such that every hyperplane meets in $\T$ in $|\Q^\pm(2n-1,q)|,|P\Q(2n-2,q)|,|\Q^\mp(2n-1,q)|$ points. But furthermore, the set of hyperplanes meeting in $|P\Q(2n-2,q)|$ or $|\Q^\pm(2n-1,q)|$ points is the same as the set of hyperplanes meeting a fixed $\Q=\Q(2n,q)$. In Lemma \ref{lem:correct-points}, we have shown that the number of hyperplanes of each type meeting a quasi-quadric is a constant. Since the hyperplanes meeting $\P\cap \pi$ in $|P\Q(2n-2,q)|$ or $|\Q^\pm(2n-1,q)|$ points respectively are meeting $\P'\cap \pi$ in $|P\Q(2n-2,q)|$ or $|\Q^\pm(2n-1,q)|$ points respectively, it follows that all hyperplanes meeting $\P\cap \pi$ in $|\Q^\mp(2n-1,q)|$ points also meet $\P'\cap \pi$ in $|\Q^\mp(2n-1,q)|$ points. By Lemma \ref{lem:equal-pointsets} we obtain that $\P\cap\pi=\P'\cap \pi$, and hence, since $\P$ and $\P'$ coincide outside $\pi$, that $\P=\P'$.

{\bf Case 2 and Case 3} are dealt with in a similar way. The proofs are contained in Appendix A.
\end{proof}

\begin{remark}\label{remark36} In \cite[Theorems 13-16]{QQ}, it is shown that for $\Q^-(2n+1,2)$ and $\Q^+(2n+1,2)$, it is possible to switch in a non-singular hyperplane and obtain a quasi-quadric which is not a quadric when $n\geq 2$ (Every hyperbolic or elliptic quasi-quadric in $\PG(3,2)$ is a quadric.). Hence, the condition $q\neq 2$ in the above theorem is necessary.

In Appendix B, we will show that there exists a quasi-quadric, which is not a quadric, obtained by switching in non-singular hyperplane of the parabolic quadrics  $Q(2n,2)$, $n\geq2$ (Proposition \ref{Q2n2}) and $\Q(2n,3)$, $n\geq 2$ (Proposition \ref{Q2n3}). We will discuss pivoting for ovals in $\PG(2,3)$ there too (Proposition \ref{Q23}).

\end{remark}

\section{Switching in singular hyperplanes}\label{sec:switching-singular}

\subsection{Switching is pivoting}\label{sec:switching-pivoting}

A proof of the following result from the theory of polar spaces can be extracted from \cite{GGG}. It will be used to prove that certain intersection sizes with hyperplanes need to preserved after switching, which severely restricts the possibilities.

\begin{lemma}\label{lem:classical2} Let $\P$ be a non-singular, non-symplectic polar space in $\PG(m,q)$ and let $\pi$ be a singular hyperplane meeting $\P$ in the cone $P\C$. The following holds:
\begin{itemize}
\item If $\nu$ is a hyperplane of $\pi$ through $P$, then the $q$ hyperplanes through $\nu$, different from $\pi$,  are all of the same type as $\nu\cap \C$.
\item If $\nu$ is a hyperplane of $\pi$, not through $P$, then
\begin{itemize}
\item If $\P$ is a quadric but not $\Q(2n,q)$, $q$ even, then there are $2$ singular and $q-1$ non-singular hyperplanes through $\nu$. Furthermore, if $\P$ is a parabolic quadric and $q$ is odd, then $(q-1)/2$ of the non-singular hyperplanes are elliptic, and $(q-1)/2$ of the non-singular hyperplanes are hyperbolic. 
\item If $\P=\Q(2n,q)$, $q$ even, and $\nu$ does not contain the nucleus $N$, then there are $q/2$ elliptic and $q/2$ hyperbolic hyperplanes through $\nu$.
\item If $\P$ is a quasi-Hermitian variety, then there are $\sqrt{q}$ singular and $q-\sqrt{q}$ non-singular hyperplanes through $\nu$, different from $\pi$.
\end{itemize}
\end{itemize}
\end{lemma}

The following lemma is the crucial technical result needed to prove that switching is pivoting.

\begin{lemma}  \label{lem:conelike}
Let $\mathcal{K}$ be a set of points in $\PG(m,q)$ and let $P_1$ and $P_2$ be two points of $\PG(m,q)$, not in $\K$, such that all hyperplanes of $\PG(m,q)$  containing neither $P_1$ nor $P_2$ contain a fixed number $s$ of points of $\mathcal{K}$. Then every plane through $P_1P_2$ meets $\K\setminus P_1P_2$ in a set of (truncated) lines through $P_i$, where depending on the plane  $i$ is either $1$ or $2$. 
Moreover the set $\K \cap P_1P_2$ contains all or none of the points of $P_1P_2$ different from $P_1$ and $P_2$. 
\end{lemma}
\begin{proof} Consider an $(m-2)$-dimensional space $\mu$ meeting the line $P_1P_2$ in a point $Q$, different from $P_1$ and $P_2$. Let $\pi=\langle \mu,P_1\rangle$, then $P_1P_2\in \pi$. Denote $|\K\cap \mu|$ by $x_\mu$. The $q$ hyperplanes through $\mu$, different from $\langle \mu,P_1\rangle$, then each contain $s-x_\mu$ points not in $\mu$. It follows that $$|\K|=|\K\cap \pi|+q(s-x_\mu).$$ 

It follows that $x_\mu$ is uniquely determined by $|\K\cap \pi|$. In other words, all hyperplanes, not through $P_1$ nor $P_2$ in the $(m-1)$-space $\pi$ contain a fixed number of points.
Repeating this argument inductively, we find that for every plane $\xi$ through $P_1P_2$, each line of $\xi$ not through $P_1$ nor $P_2$ meets $\K$ in a fixed number, say $y$ points. Note that $y$ as well as other quantities below depend on $\xi$ but we omit this in our notation.

Consider the points in $\xi\cap \K$, not contained in the line $P_1P_2$. We first show that either all points of $P_1P_2$, different from $P_1$ and $P_2$ belong to $\K$, or none of those belong to $\K$.
Counting incident pairs $(P,L)$, where $P$ is a point of $\K$ and $L$ is a line not through $P_1$ or $P_2$, yields:
$$(q^2-q)y=|\K\cap P_1P_2|q+|\K\setminus (\K\cap P_1P_2)|(q-1).$$
Since the left hand side is a multiple of $q-1$, it follows that $\K\cap P_1P_2$ is a multiple of $q-1$. So $\K\cap P_1P_2=0$ or $q-1$ as claimed. 

This implies that every line, not through $P_1$ nor $P_2$ has a fixed number, say $x$, of points of the set $\K':=(\xi\cap\K)\setminus (\K\cap P_1P_2)$ (note that $x=y$ or $x=y-1$).

Let $Q$ be a point of $P_1P_2$, different from $P_1$ and $P_2$. Since every line in $\xi$ through $Q$ contains $x$ points of $\K'$, we have that $|\K'|=qx$. 

If $\K'$ does not consist of the points of a set of lines through $P_1$ with $P_1$ removed, then
we find a point $R$ of $\K'$ and a point $S\notin \K'$, not on $P_1P_2$, on $RP_1$. 

Counting points of $\K'$ on lines through $R$ yields

$$(q-1)(x-1)+|P_1R\cap \K'|+|P_2R\cap \K'|-1=qx,$$
and similarly for $S$, we find
$$(q-1)x+|P_1S\cap \K'|+|P_2S\cap \K'|=qx.$$
It follows that $|P_1R\cap \K'|+|P_2R\cap \K'|=q+x$ and $|P_1S\cap \K'|+|P_2S\cap \K'|=x$. Since $P_1R=P_1S$, it follows that

 $$|P_2R\cap \K'|-|P_2S\cap \K'|=q.$$
 
 Now recall that $P_2\notin \K'$ so $|P_2R\cap \K'|\leq q$. It follows that $|P_2R\cap \K'|=q$ and $|P_2S\cap \K'|=0$.
 
 We see that for every point $R'$ in $\K$ on $P_1R$, the points of $R'P_2\setminus \{P_2\}$ belong to $\K'$, and for every point $S'\notin \K'$, there are no points of $\K'$ on the line $S'P_2$.
 It follows that $\K'$ consists of the point set of a set of lines through $P_2$, with $P_2$ removed.
Since $\K\cap \xi$ consists of the set $\K'$ with all or none of the points of $P_1P_2$ different from $P_1$ and $P_2$ added, the lemma is proved.
\end{proof}

\begin{corollary} \label{lem:sing-char} Let $\mathcal{K}$ be a set of points in $\PG(m,q)$ and let $P$ be a point of $\PG(m,q)$ such that all hyperplanes of $\PG(m,q)$ not containing $P$ contain a fixed number $s$ of points of $\mathcal{K}$. Then $\mathcal{K}\cup \{P\}$ consists of the $qs+1$ points of a cone with vertex $P$ and base a set of $s$ points.
\end{corollary} 
\begin{proof} If $Q$ is an arbitary point of $\PG(m,q)$, then applying Lemma \ref{lem:conelike} to $P_1=P$ and $P_2=Q$ shows that $\K$ is a union of lines through $P$ and $Q$ (without $P,Q$). Since this holds for every choice of $Q$, we find that $\K$ consist of lines through  $P$ only.
\end{proof}

Recall that there are two different intersection sizes for hyperplanes with quasi-polar spaces of elliptic, hyperbolic of Hermitian type. In what follows, we will use the following convention for these sizes in $\PG(m,q)$: we let $A_{m-1}$ and $B_{m-1}$ denote the sizes of their intersection with hyperplanes, where we assume that $A_{m-1}<B_{m-1}$. More precisely, if the quasi-polar space $\P$ in $\PG(m,q)$ is
\begin{itemize}
\item elliptic, then $A_{m-1}=|P\Q^-(m-2,q)|$ and $B_{m-1}=|\Q(m-1,q)|$,
\item hyperbolic, then $A_{m-1}=|\Q(m-1,q)|$ and $B_{m-1}=|P\Q^+(m-2,q)|$,
\item quasi-Hermitian with $m$ even, then $A_{m-1}=|P\H(m-2,q)|$ and $B_{m-1}=|\H(m-1,q)|$,
\item quasi-Hermitian with $m$ odd, then $A_{m-1}=|\H(m-1,q)|$ and $B_{m-1}=|P\H(m-2,q)|$.
\end{itemize}
It can be checked that in all cases $B_{m-1}-A_{m-1}=q(B_{m-3}-A_{m-3})$.

\begin{proposition} \label{prop:switchispivot} Let $\P$ be a polar space of elliptic, hyperbolic or Hermitian type in $\PG(m,q)$, $q$ arbitary, or of parabolic type in $\PG(m,q)$, $q$ odd, $m\geq 3$.
Suppose that $\P'$ is a quasi-polar space with $|\P|=|\P'|,$ obtained by switching in a singular hyperplane $\pi$ of $\P$ with vertex $P$. Then $\P'$ is the point set of a cone with vertex $P$ over a quasi-polar space of the same type as $\mathcal{P}$. 
\end{proposition}

\begin{proof}  Let $\mu$ be a hyperplane of $\pi$, not through $P$. Then $\mu$ meets the cone $\P\cap \pi$ in a non-singular polar space of the same type as $\P$. 
If $\P$ is of elliptic, hyperbolic of Hermitian type, then there are two types of hyperplanes. By Lemma \ref{lem:classical2}, $\mu$ lies on hyperplanes, different from $\pi$, of both types. Since $\P$ and $\P'$ coincide outside $\pi$, and $\P$ and $\P'$ need to have the same intersection sizes with hyperplanes, it follows that $|\mu\cap\P'|=|\mu\cap \P|$. Since this holds for all hyperplanes $\mu$ not through $\pi$, and $|\mu\cap \P|$ is independent of the choice of $\mu$, we find by
Corollary \ref{lem:sing-char} that $\P'$ is indeed the point set of a cone with vertex $P$. Note that $P$ needs to be included in $\P'$ since $|\P\cap \pi|=|\P'\cap \pi|=1\pmod{q}$ and the set of points of a cone, without the vertex $P$ has size $0\pmod{q}$. 

Suppose first that $\P$ is of elliptic, hyperbolic of Hermitian type. Since we already know that $\P'$ is the point set of a cone with vertex $\P$, in order to prove the proposition, we only need to show that for all hyperplanes $\nu$ of $\pi$ through $P$, $|\mu\cap \P'|\in \{qA_{m-3}+1,qB_{m-3}+1\}$.

Consider a hyperplane $H$ of $\PG(m,q)$, different from $\pi$, through $\nu$ and suppose that $|\nu\cap\P|=qA_{m-3}+1$. By Lemma \ref{lem:classical2}, $|H\cap \P|=A_{m-1}$ and it follows that $|H\cap \P'|=|(H\setminus \pi)\cap \P'|+|\mu\cap \P'|=|(H\setminus \pi)\cap \P|+|\mu\cap \P'|=A_{m-1}-(qA_{m-3}+1)+|\mu\cap \P'|$. Since $|H\cap \P'|\in \{A_{m-1},B_{m-1}\}$, it follows that either 
$|\mu\cap \P'|=qA_{m-3}+1$ or $|\mu\cap \P'|=B_{m-1}-A_{m-1}+qA_{m-3}+1=qB_{m-3}+1$.
If $|H\cap \P|=B_{m-1}$, we obtain the same conclusion, using again that $B_{m-1}-A_{m-1}=q(B_{m-3}-A_{m-3})$.

Now assume that $\P$ is a non-singular quadric $\Q(2n,q)$, $q$ odd. Let $\nu$ be a hyperplane of $\pi$ through $P$ and let $2n=m$, $A_{m-1}=|\Q^-(m-1,q)|$, $B_{m-1}=|P\Q(m-2,q)|$, $C_{m-1}=|\Q^+(2n-1,q)$. Since we already know that $\P'$ is the point set of a cone with vertex $\P$, and $|\P|=|\P'|$, in order to prove the proposition, we only need to show that for all hyperplanes $\nu$ of $\pi$ through $P$, $|\mu\cap \P'|\in \{qA_{m-3}+1,qB_{m-3}+1,qC_{m-3}+1\}$.

Consider a hyperplane $H$ of $\PG(m,q)$, different from $\pi$, through $\nu$ and suppose that $|\nu\cap\P|=qA_{m-3}+1$. By Lemma \ref{lem:classical2}, $|H\cap \P|=A_{m-1}$ and it follows that $|H\cap \P'|=|(H\setminus \pi)\cap \P'|+|\mu\cap \P'|=|(H\setminus \pi)\cap \P|+|\mu\cap \P'|=A_{m-1}-(qA_{m-3}+1)+|\mu\cap \P'|$. Since $|H\cap \P'|\in \{A_{m-1},B_{m-1},C_{m-1}\}$, it follows that either 
$|\mu\cap \P'|=qA_{m-3}+1$ or $|\mu\cap \P'|=B_{m-1}-A_{m-1}+qA_{m-3}+1=qB_{m-3}+1$, or $|\mu\cap \P'|=C_{m-1}-A_{m-1}+qA_{m-3}+1=qC_{m-3}+1$.
Similarly, if $|\nu\cap\P|=qB_{m-3}+1$, then $|\mu\cap \P'|=qB_{m-3}+1$ or $|\mu\cap \P'|=A_{m-1}-B_{m-1}+qB_{m-3}+1=qA_{m-3}+1$, or $|\mu\cap \P'|=C_{m-1}-B_{m-1}+qB_{m-3}+1=qC_{m-3}+1$. Finally,  if $|\nu\cap\P|=qB_{m-3}+1$, then $|\mu\cap \P'|=qB_{m-3}+1$ or $|\mu\cap \P'|=A_{m-1}-B_{m-1}+qB_{m-3}+1=qA_{m-3}+1$, or $|\mu\cap \P'|=C_{m-1}-B_{m-1}+qB_{m-3}+1=qC_{m-3}+1$.
\end{proof}

Proposition \ref{prop:switchispivot} does not treat the case $\Q(2n,q)$, $q$ even. For parabolic quadrics in even characteristic, it is no longer true that pivoting in a singular hyperplane $\pi$ can only be done by replacing $\P\cap \pi$ with a cone over a quasi-quadric. 
For $\Q(4,q)$, we will be able to describe exactly what other possibilities we have for $\P'\cap \pi$ (see Proposition \ref{prop:Q4q}) whereas for $\Q(2n,q)$, $n>2$, $q$ even, we find a strong restriction on the set $\P'\cap \pi$ but no full classification. We provide an example of switching in a singular hyperplane which is not pivoting in $\Q(2n,q)$, $q$ even, in Example \ref{ex:switchnotpivot}.
 
In the following propositions, we use the terminology {\em truncated} line through a point $P$ to denote the set of points on that line, except the point $P$. Similarly, a {\em truncated cone} with vertex a point $P$ and base a point set $\C$ consists of the point set of the cone $P\C$ where the point $P$ is removed.
\begin{proposition}\label{prop:conelike}
Let $\P=\Q(2n,q)$, $q$ even, $n\geq 2$, and suppose that $\P'$ is a parabolic quasi-quadric with $|\Q(2n,q)|$ points, obtained by switching in a singular hyperplane $\pi$ of $\P$ with vertex $P$. 
Then $\P'\cap \pi$ is the point set of $\frac{q^{2n-2}-1}{q-1}$ truncated lines through $P$ or $N$, to which either all points of the line $PN$ or exactly one of the points $P$ and $N$ is added.
The only points lying on more than one line are possibly $P$ and $N$.
Furthermore, every hyperplane through $P$ or $N$ meets $\P'$ in  $|P\Q^-(2n-3,q)|,|\Q(2n-2,q)|$, or $|P\Q^+(2n-3,q)|$ points. 

Conversely, the point set of any set of $\frac{q^{2n-2}-1}{q-1}$ truncated lines through $P$ or $N$, to which either all points of the line $PN$ or exactly one of the points $P$ and $N$ is added, which satisfies the property that every hyperplane through $P$ or $N$ meets $\P'$ in  $|P\Q^-(2n-3,q)|,|\Q(2n-2,q)|$, or $|P\Q^+(2n-3,q)|$ points, gives rise to a parabolic quasi-quadric.
\end{proposition}
\begin{proof} Let $\mu$ be a hyperplane of $\pi$, not through $P$ or $N$. Then $\mu$ meets the cone $\P\cap \pi$ in a $\Q(2n-2,q)$, and by Lemma \ref{lem:classical2}, $\mu$ lies on at least one hyperbolic and one elliptic hyperplane. This implies that $|\mu\cap \P'|=|\mu\cap \P|=|\Q(2n-2,q)|$. It follows from Lemma \ref{lem:conelike} that $\P'$ is the set $\mathcal{L}$ of (truncated) lines through $P$ and $N$. Since $|\P'|=1\mod q$, we have that if $PN$ is not a line of this set, exactly one of $P$ and $N$ are in $\P'$, and there are exactly $\frac{q^{2n-2}-1}{q-1}$ lines in $\mathcal{L}$.

Now consider a hyperplane $\nu$ of $\pi$ through $N$. All hyperplanes through $\nu$ are singular, and hence, a hyperplane different from $\pi$ through $\nu$ has precisely $|P\Q(2n-2,q)|-|\Q(2n-2,q)|=q^{2n-2}$ points of $\P$, not in $\pi$. It follows that $\nu\cap \P'$ has $|P\Q^-(2n-3,q)|,|\Q(2n-2,q)|$, or $|P\Q^+(2n-3,q)|$ points.

A hyperplane $\xi$ of $\pi$ through $P$ contains $|P\Q^-(2n-3,q)|,|\Q(2n-2,q)|$, or $|P\Q^+(2n-3,q)|$ points of $\P$, and by Lemma \ref{lem:classical2}, it follows that, in all these cases, every hyperplane, different from $\pi$ through $\xi$ has the same number, namely $q^{2n-2}$, of points of $\P$ not in $\pi$. It again follows that $\nu\cap \P'$ has $|P\Q^-(2n-3,q)|,|\Q(2n-2,q)|$, or $|P\Q^+(2n-3,q)|$ points.

It is clear from the above reasoning that any set $\S$ of $|P\Q(2n-2,q)|$ points such that all hyperplanes not through $P$ or $N$ contain $|\Q(2n-2,q)|$ points of $\S$ and all hyperplanes through $P$ or $N$ contain $|P\Q^-(2n-3,q)|,|\Q(2n-2,q)|$, or $|P\Q^+(2n-3,q)|$ points of $\S$ satisfy the property that every hyperplane meets the set $\P\setminus (\P\cap \pi) \cup \S$ in $|\Q^-(2n-1,q)|,|P\Q(2n-2,q)|$, or $|\Q^+(2n-1,q)|$ points, and hence, gives rise to a parabolic quasi-quadric.
\end{proof}

\begin{corollary}  \label{prop:Q4q} Let $\P=\Q(4,q)$, $q$ even, and suppose that $\P'$ is a parabolic quasi-quadric with $|\Q(4,q)|$ points, obtained by switching in a singular hyperplane $\pi$ of $\P$ with vertex $P$. Then $\P'\cap \pi$ is one of the following: \begin{itemize}
\item a cone with vertex $P$ and base an oval;
\item a cone with vertex $N$ and base an oval;
\item the union of a truncated cone with vertex $P$ and base a $q$-arc together with one line through $N$ disjoint from the truncated cone, and different from $PN$;
\item the union of a truncated cone with vertex $N$ and base a $q$-arc together with one line through $P$ disjoint from the truncated cone, and different from $PN$,
\end{itemize}
and for all of the above possibilities, $\P'$ is indeed a parabolic quasi-quadric.
\end{corollary}

\begin{proof}

By Proposition \ref{prop:conelike}, we have that $\P'$ consists of the point set of $q+1$ truncated lines $\L$, to which either the points of the line $PN$, or exactly one of the points $P$ and $N$ are added.

Note that the roles of $P$ and $N$ are the same, so assume without loss of generality that $P\in \P'$. 

If all lines of $\L$ go through $P$ then we get a cone with vertex $P$. Let $\mu$ be a plane not through $P$ or $N$, then $\mu$ meets the lines of $\L$ in a set $\S$ of $q+1$ points. Since every plane through $P$ has at most $2q+1$ points of $\P'$, we have that every line of $\mu$ has at most $2$ points of $\S$; so $\S$ is an oval. 

If there is exactly one line, say $L$, through $P$, different from $PN$, contained in $\L$, then $\L$ is union of $L$ with a cone with vertex $N$. Since every plane through $N$ contains at most $2q+1$ points, the set $C$ meets every line of $\mu$ in at most $2$ points, and hence, $C$ is a $q$-arc.

Assume thus there are two lines, $L_1,L_2$ through 
$P$ contained in $\L$ but not all $q+1$ lines of $\L$ go through $P$. Then there is a line $L_3$ through $N$, different from $PN$, such that $L_3\setminus \{N\}$ is contained in $\P'$. But this implies that the plane $\langle L_1,L_2\rangle$ contains at least $2q+1+1$ points of $\P'$, namely those on $L_1,L_2$ together with $N\cap \langle L_1,L_2\rangle$, a contradiction.
By reversing the roles of $P$ and $N$, we see that $\P'$ is one of the four possibilities of the statement.

Finally, it is easy to see that all the above sets have the property that every plane not through $P$ or $N$ meets this set in $q+1$ points, while the planes through $P$ and $N$ meet in $1,q+1$ or $2q+1$ points. We will show that for all solids $H$, $|H\cap \P'|\in \{q^2+1,q^2+q+1,q^2+2q+1\}$.

It is clear that all solids $H$ not through $P$ or $N$ have $|H\cap \pi\cap \P|=|H\cap\pi\cap\P'|$ which implies $|H\cap\P|=|H\cap\P'|$.
By Lemma \ref{lem:classical2}, we see that all solids, different from $\pi$, though $P$ have $q^2$ points of $\P$, not in $\pi$. Likewise, all solids different from $\pi$ through $N$ have $q^2$ points of $\P$ outside $\pi$: all those solids are singular, and hence, meet $\P$ in $q^2+q+1$ points, and all planes through $N$ in $\pi$ have $q+1$ points of $\P$.
 Since $|\P'\cap\pi|\in \{1,q+1,2q+1\}$ we find that all solids through $P$ or $N$ meet $\P'$ in $q^2+1,q^2+q+1$ or $q^2+2q+1$ points as desired. Finally, we see that if $H=\pi$, then $|H\cap \P'|=q^2+q+1$.
\end{proof}

\begin{example} \label{ex:switchnotpivot}  Consider a hyperplane $\pi$ intersecting $\P=\Q(2n,q)$, $q$ even, $n\geq 2$, in a cone $P\Q(2n-2,q)$, where we take the base to be contained in a hyperplane $\mu$ of $\pi$, not through $P$. Let $\nu_P$ be an $(n-2)$-space contained in the base $\Q(2n-2,q)$, and let $\nu_N$ be an $(n-2)$-space of $\mu$, disjoint from $\Q(2n-2,q)\setminus \nu_P$. Note that $\nu_N$ always exists since $\nu_P$ is contained in a unique $(n-1)$-space which intersects $\Q(2n-2,q)$ exactly in $\nu_P$. To see this, consider the quotient with a hyperplane $\pi_H$ of $\nu_P$: there is a unique tangent line through the point $\nu_P/\pi_H$ to the conic $\Q(2n-2,q)/\pi_H$. 

Let $\S$ be the set $(P\Q(2n-2,q)\setminus P\nu_P)\cup P\nu_N$. Then it is easy to check that this set $\S$ satisfies the conditions of Proposition \ref{prop:conelike}: any hyperplane of $\pi$ not through $P$ and $N$ clearly has $|\Q(2n-2,q)|$ points of $\S$.  Any hyperplane of $\pi$ containing both $P\nu_P$ and $N\nu_N$ has the same intersection size with $\S$ as with $P\Q(2n-2,q)$, and hence, meets $\S$ in $|P\Q^-(2n-3,q)|,|\Q(2n-2,q)|$, or $|P\Q^+(2n-3,q)|$ points.
If a hyperplane $\Sigma$ contains $P\nu_P$ but not $N\nu_N$, it meets $N\nu_N$ in an $(n-2)$-space. Since $\Sigma$ contains $P\nu_P$, it meets $\P$ in $|\Q(2n-2,q)|$ or $|P\Q^+(2n-3,q)|$ points, so it meets in $\S$ in $|\Q(2n-2,q)|-|P\nu_P|+\frac{q^{n-1}-1}{q-1}=|P\Q^-(2n-3,q)|$ or $|P\Q^+(2n-3,q)|-|P\nu_P|+\frac{q^{n-1}-1}{q-1}=|\Q(2n-2,q)|$ points. Similarly, if $\Sigma$ does not contain $P\nu_P$ but does contain $N\nu_N$, it meets $\S$ in $|P\Q^-(2n-3,q)|-\frac{q^{n-1}-1}{q-1}+|N\nu_N|=|\Q(2n-2,q)|$ or $|\Q(2n-2,q)|-\frac{q^{n-1}-1}{q-1}+|N\nu_N|=|P\Q^+(2n-3,q)|$ points.

\end{example}

Parabolic quasi-quadrics in even characteristic will be looked at in more detail in Section \ref{PQQ}.

\subsection{Switching in quasi-quadrics}

\begin{theorem} \label{prop:switching-is-pivoting}
Let $\P$ be a quasi-polar space of elliptic, hyperbolic or Hermitian type in $\PG(m,q)$, $m\geq 3$, where $q$ is square in the case $\P$ is Hermitian, or $\P$ is a parabolic quasi-quadric in $\PG(2n,q)$, $n\geq 2$, of size $|\Q(2n,q)|$. Suppose that $\P$ contains a hyperplane $\pi$ such that $\pi\cap \P$ is a cone with vertex a point $P$ and base a quasi-polar space $\tilde{\P}$ of the same type as $\P$ in a hyperplane $\mu$ not through $P$ of $\pi$.
Let $\tilde{\P'}$ be a quasi-polar space of the type of $\P$ in $\mu$, then the set $\P'=(\P\setminus P\tilde{\P})\cup P\tilde{\P}'$ is a quasi-polar space of the type of $\P$.
\end{theorem}
\begin{proof}
{\bf Case 1:} Let $H$ be a hyperplane of $\PG(m,q)$. We need to show that $H\cap \P'\in \{A_{m-1},B_{m-1}\}$. Note that $\P\cap (H\setminus \pi)=\P'\cap (H\setminus \pi)$. 

If $H$ does not contain $P$, then $H\cap \pi$ has $A_{m-2}$ points of both $\P$ and $\P'$, so $|H\cap \P'|=|H\cap \P|\in \{A_{m-1},B_{m-1}\}$.
So suppose that $H$ contains $P$, then, since every hyperplane of $\mu$ meets the base $\tilde{\P}$ in either $A_{m-3}$ or $B_{m-3}$ points, $H\cap \pi$ meets in $\P$ in $qA_{m-3}+1$ or $qB_{m-3}+1$ points. Likewise, $H\cap \pi$ meets $\P'$ in $qA_{m-3}+1$ or $qB_{m-3}+1$ points.
If $|H\cap \pi\cap\P|=qA_{m-3}+1$, then $|H\cap P|=A_{m-1}$ by Lemma \ref{lem:classical2}. We find that $|H\cap \P'|$ is either $A_{m-1}$ or $A_{m-1}-(qA_{m-3}+1)+(qB_{m-3}+1)$. Since $qB_{m-3}-qA_{m-3}=B_{m-1}-A_{m-1}$, we indeed have that $|H\cap \P'|\in \{A_{m-1},B_{m-1}\}$. Similarly, if $|H\cap \pi\cap\P|=qB_{m-3}+1$, by Lemma \ref{lem:classical2}, $|H\cap \P|=B_{m-1}$, and it follows that $|H\cap \P'|$ is either $B_{m-1}$ or $B_{m-1}-qB_{m-3}+qA_{m-3}=A_{m-1}$.

{\bf Case 2:} $\P$ is a parabolic quasi-quadric in $\PG(2n,q)$ of size $|\Q(2n,q)|$. The proof is contained in Appendix A and is similar to that of Case 2, this time with three intersection numbers.
\end{proof}

\section{Repeated pivoting}\label{sec:repeated}

In \cite[page 19]{QD}, the authors write that ``one can repeat pivoting as much as one wants, implying that the family of quasi-quadrics is quite wild''. But since they define pivoting as cone replacement, it is not clear that one can repeat this procedure as much as one wants: after each pivot, in order to be able to pivot again, there still should be a hyperplane meeting the obtained quasi-quadric in a cone. In fact, \cite{QD} does not show how to pivot more than once. In Proposition \ref{repeatedpivoting} we give a construction enabling us to pivot $q+1$ times.

Recall that if $\perp$ denotes the polarity associated with a polar space $\P$ in $\PG(m,q)$ and $P$ is a point of $\P$, then $P^\perp\cap \P$ is a cone $P\C_P$, where $\C_P$ is a classical polar space of the same type as $\P$ in $\PG(m-2,q)$. 

While there is no orthogonal polarity associated to $\Q=\Q(2n,q)$, $q$ even, we still have that for every point $P\in \Q$, there is a unique hyperplane $\pi_P$ containing all points which are collinear with $P$ in $\Q$, and these form a cone $P\C$ where $\C$ is a non-singular parabolic quadric contained in a hyperplane of $\pi_P$. We will, by abuse of notation, denote this hyperplane $\pi_P$ by $P^\perp$. Note that the nucleus $N$ of $\Q$ is contained in all spaces $P^\perp$. 

Using this definition for hyperplanes $P^\perp$ defined for points on a parabolic quadric, $q$ even, the following now holds for all non-singular polar spaces $\P$ in $\PG(m,q)$: if $Q$ is a point collinear with $P$ in $\P$, then $P^\perp\cap Q^\perp$ is an $(m-2)$-dimensional space meeting $\P$ in a cone $L\tilde{\P}$ with vertex the line $L=PQ$ and base a non-singular polar space $\tilde{\P}$ of the same type as $\P$ contained in an $(m-4)$-dimensional subspace of $P^\perp\cap Q^\perp$, disjoint from $PQ$. The $q+1$ hyperplanes of $\PG(m,q)$ through $P^\perp\cap Q^\perp$ are precisely those of the form $S^\perp$ for a point $S$ on the line $PQ$.

The following lemma is a classical result, so we omit the proof, which is (hidden in) \cite{GGG}.

\begin{lemma}\label{lem:classical1}
Let $\P$ be a non-singular, non-symplectic polar space contained in $\PG(m,q)$ and let $\pi$ be a hyperplane of $\PG(m,q)$. Let $\mu$ be an $(m-3)$-dimensional subspace of $\pi$ such that $\P\cap\mu$ is a cone $P\C$, where $\C$ is a non-singular polar space of the same type as $\P$.

Then all but one hyperplane of $\pi$ through $\mu$ intersects $\P$ in a non-singular polar space of the same type as $\P$. The last hyperplane, $P^\perp$ meets $\P$ as follows:
\begin{itemize}
\item if $\P=\Q^\pm(2n+1,q)$, then $P^\perp\cap\P$ is a cone $P\Q(2n-2,q)$ if $\P\cap H=\Q(2n,q)$ and a cone $L\Q^\pm(2n-3,q)$ if $\P\cap H=P\Q^\pm(2n-1,q)$.
\item if $\P=\H(m,q)$, then $P^\perp\cap\P$ is a cone $P\H(m-3,q)$ if $\P\cap H=\H(m-1,q)$ and a cone $L\H(m-4,q)$ if $\P\cap H=P\H(m-2,q)$.
\item if $\P=\Q(2n,q)$, then $P^\perp\cap\P$ is a cone $P\Q^\pm(2n-3,q)$ if $\P\cap H=\Q^\pm(2n-1,q)$, and a cone $L\Q(2n-4,q)$ if $H\cap \P=P\Q(2n-2,q)$.
\end{itemize}
\end{lemma}

The following proposition shows that we can construct a quasi-polar space by pivoting $q+1$ times. This construction can be thought of as repeated pivoting in the singular hyperplanes defined by each of the $q+1$ points on a fixed line. But since we are not modifying the intersection of any two of these singular hyperplanes, these $q+1$ pivots can be done simultaneously and their order does not matter.

\begin{proposition} \label{repeatedpivoting}Let $\P$ be a non-singular, non-symplectic polar space in $\PG(m,q)$. Let $P$ and $Q$ be two collinear points in $\P$ and let $\xi$ be the $(m-2)$-space $P^\perp\cap Q^\perp$.

For every point $R\in PQ$, consider a cone $R\C'_R$ contained in $R^\perp$, such that $R\C'_R\cap \xi=R^\perp \cap \P \cap \xi$, and such that $\C'_R$ is a polar space of same type as $\P$.
Then the set $\P'=\cup_{R\in PQ}R\C'_R$ is a quasi-polar space (of the same type as $\P$).
\end{proposition}
\begin{proof} 

Let $H$ be a hyperplane of $\PG(m,q)$. We need to show that $|H\cap \P'|$ is one of the intersection numbers of hyperplanes with respect to $\P$. We distinguish $3$ cases, according to whether $H$ contains $\xi$, $H$ meets $\xi$ in a hyperplane containing the line $PQ$, or $H$ meets $\xi$ in a hyperplane not containing the line  $PQ$.

{\bf Case 1:} $H$ contains $\xi$.

In this case, we have seen that $H=S^\perp$ for some $S\in PQ$, so $|H\cap \P'|=|S\C'_S|=|S\C_S|=|H\cap \P|$.

{\bf Case 2:} $H$ does not contain the line $PQ$.
Let $R$ be the unique intersection point of $PQ$ with $H$. Note that  $$|H\cap \P'|=\sum_{S\in PQ}|(H\cap (S^\perp\setminus \xi))\cap \P'|+|(H\cap\xi)\cap \P'|.$$

By our construction, $\xi\cap \P=\xi\cap\P'$, and hence, $|(H\cap\xi)\cap \P'|=|(H\cap\xi)\cap \P|$.
We also see that $|(H\cap S^\perp)\cap\P'|=|(H\cap S^\perp)\cap\P|$ for all points $S\neq R$ in $PQ$: the hyperplane $S^\perp$ meets both $\P$ and $\P'$ in a cone with vertex $S$ and base a polar space $\C_S$ and $\C'_S$ of the same type as $\P$. Since $H\cap S^\perp$ is a hyperplane of $S^\perp$, not through the vertex $S$, it indeed meets both cones in $|\C_S|=|\C'_S|$ points. Since, $\xi\cap \P=\xi\cap\P'$, we have that $|H\cap (S^\perp\setminus \xi))\cap \P'|=|(H\cap (S^\perp\setminus \xi))\cap \P|$.

Now consider the intersection $H\cap R^\perp\cap\P'$.
We know that $R^\perp\cap \P$ and $R^\perp\cap \P'$ are cones with vertex $R$ and base a polar space $\tilde{\P}$, $\tilde{\P'}$, resp. of the same type as $\P$ in a hyperplane $\sigma$ of $R^\perp$. 
Suppose first that $\P$ is of elliptic, hyperbolic or Hermitian type, then there are two intersection numbers of hyperplanes with respect to $\P$. As before, let these be denoted by $A_{m-1}$ and $B_{m-1}$ with $A_{m-1}<B_{m-1}$. Every hyperplane of $\sigma$ then meets the base $\tilde{\P}$ in either $A_{m-3}$ or $B_{m-3}$ points and since $R\in H$, $(H\cap R^\perp)\cap \P$ then has two intersection sizes, namely $qA_{m-3}+1$ and $qB_{m-3}+1$.
If $|H\cap \P|=A_{m-1}$, then by Lemma \ref{lem:classical1}, $H\cap R^\perp$ meets the cone $R\tilde{\P}$ in $qA_{m-3}+1$ points. Hence, if $|H\cap R^\perp\cap \P|\neq |H\cap R^\perp\cap \P'|$, then $|H\cap R^\perp\cap \P'|=qB_{m-3}+1$. It follows that $|H\cap \P'|=|H\cap \P|+qB_{m-3}-qA_{m-3}$. Since $qB_{m-3}-qA_{m-3}=B_{m-1}-A_{m-1}$, it follows that $|H\cap \P'|=B_{m-1}$. Similarly, if $|H\cap \P|=B_{m-1}$ and $|H\cap \P'|\neq B_{m-1}$, then $|H\cap \P'|=A_{m-1}$.

If $\P$ is of parabolic type, then $\P=\Q(2n,q)$ and all hyperplanes meet in $A_{2n-1}=|\Q^-(2n-1,q)|$, $B_{2n-1}=|P\Q(2n-2,q)|$, or $C_{2n-1}=|\Q^+(2n-1,q)|$ points. As above, we know that $R^\perp\cap \P$ and $R^\perp\cap \P'$ are cones with vertex $R$ and base a polar space $\Q(2n-2)$ in a hyperplane $\sigma$ of $R^\perp$. 
We see that every hyperplane of $\sigma$ then meets the base $\tilde{\P}$ in either $A_{2n-3}$, $B_{2n-3}$, or $C_{2n-3}$ points and since $R\in H$, $(H\cap R^\perp)\cap \P$ then has three intersection sizes, namely $qA_{2n-3}+1$, $qB_{2n-3}+1$ and $qC_{2n-3}+1$.
If $|H\cap \P|=A_{2n-1}$, then by Lemma \ref{lem:classical1}, $H\cap R^\perp$ meets the cone $R\tilde{P}$ in $qA_{2n-3}+1$ points. Hence, if $|H\cap R^\perp\cap \P|\neq |H\cap R^\perp\cap \P'|$, then $|H\cap R^\perp\cap \P'|\in \{qB_{2n-3}+1,qC_{2n-3}+1\}$. It follows that $|H\cap \P'|=|H\cap \P|+qB_{2n-3}-qA_{2n-3}$ or $|H\cap \P'|=|H\cap \P|+qC_{2n-3}-qA_{2n-3}$. Since $qC_{2n-3}-qA_{2n-3}=C_{2n-1}-A_{2n-1}$, it follows that $|H\cap \P'|=C_{2n-1}$. The cases where $|H\cap \P|=B_{2n-1}$ and $|H\cap \P|=C_{2n-1}$ follow by an analogous reasoning. We conclude that $|H\cap \P'|\in \{A_{2n-1},B_{2n-1},C_{2n-1}\}$.

{\bf Case 3:} $H\cap \xi$ is a hyperplane of $\xi$ containing $PQ$.
Recall that $$|H\cap \P'|=\sum_{S\in PQ}|(H\cap (S^\perp\setminus \xi))\cap \P'|+|(H\cap\xi)\cap \P'|.$$ We claim that for each $R\in PQ$, $|(H\cap(R^\perp\setminus \xi))\cap \P'|=|(H\cap(R^\perp\setminus \xi))\cap \P|$.
We have that $R^\perp$ meets both $\P$ and $\P'$ in a cone with vertex $R$ and base a polar space $\C_R$ and $\C'_R$ of the same type as $\P$. These bases, $\C_R$ and $\C'_R$ are contained in a hyperplane $\sigma$ of $R^\perp$ which meets $\xi$ in a hyperplane $\mu$ of $\sigma$, not through $R$. Let $T=PQ\cap \mu$. By our assumption that $R\C'_R\cap \xi=R^\perp \cap \P \cap \xi$, we have that $\mu\cap \C_R=\mu\cap \C'_R$. Moreover, since $\mu$ is contained in $\xi=P^\perp\cap Q^\perp$, $\mu\subset T^\perp$, and hence, $\mu\cap \C_R$ is a cone with vertex $T$ and base a polar space of the same type as $\P$.

The hyperplane $H$ meets $\mu$ in an $(m-4)$-space $\nu$ of $\mu$ through the vertex $T$. Since $H$ contains the vertex $R$ of the cone $R\C_R$, showing our claim that $|(H\cap(R^\perp\setminus \xi))\cap \P'|=|(H\cap(R^\perp\setminus \xi))\cap \P|$ is equivalent to showing that the hyperplane $H\cap\sigma$ of $\sigma$, which meets $\xi$ in the $(m-4)$-space $\nu$, meets $\C_R$ and $\C'_R$ in the same number of points. Recall that $\xi\cap \P=\xi\cap \P'$, and hence, $\nu\cap \P=\nu\cap \P'$.
The $(m-4)$-space $\nu$ lies on one singular $(m-3)$-space in $\sigma$, namely $\mu$, all other $(m-3)$-spaces through $\nu$ in $\sigma$ are of the same type for both $\P$ and $\P'$ by Lemma \ref{lem:classical2}. This proves our claim and we conclude that $|H\cap \P|= |H\cap \P'|$.
\end{proof}

\begin{remark} Not all quasi-quadrics can be obtained by applying repeated switching to a quadric. For example, it follows from Lemma \ref{prop:switchispivot} that switching in $\Q^-(3,q)$ is not possible (since a singular hyperplane of $\Q^-(3,q)$ consists of a cone with vertex $P$ and base the empty set). This shows that Suzuki-Tits ovoids cannot be obtained from an elliptic quadric $\Q^-(3,q)$ by repeated switching.
\end{remark}

\section{Parabolic quasi-quadrics in even characteristic}\label{PQQ}

\subsection{Definition of De Clerck-Hamilton-O'Keefe-Penttila} \label{sec:FDC}
Recall that we have defined a parabolic quasi-quadric in $\PG(2n,q)$ as a set $\S$ of points such that hyperplanes intersect $\S$ in $|\Q^+(2n-1,q)$, $|\Q^-(2n-1,q)|$ or $|P\Q(2n-2,q)|$ points.
Parabolic quasi-quadrics have three different intersection numbers with respect to hyperplanes, and as such, behave differently from the other quasi-polar spaces. In particular, the size of a parabolic quasi-quadric does not follow from the definition (unlike in the other cases, see Lemma \ref{lem:correct-points} and Remark \ref{rem:sizeparabolic}). Parabolic quadrics in $\PG(2n,q)$, $q$ even, always have a {\em nucleus}. This is a point $N$, not contained in the quadric, such that all lines through $N$ intersect the parabolic quadric in exactly one point (see \cite[Corollary 1.8 (i)]{GGG}).

In \cite{QQ}, when introducing parabolic quasi-quadrics for $q$ even, the authors explicitely ask for the size to be that of a parabolic quadric and for the existence of a point that acts as a nucleus.

More precisely, they define a {\em parabolic quasi-quadric with nucleus $N$} in $\PG(2n,q)$, $q$ even, to be a set $\S$ of points such that 

\begin{itemize}
\item[(a)] $|S| = \frac{q^{2n}-1}{q-1}$;
\item[(b)] Every hyperplane not through $N$ intersects $\S$ in $|\Q^-(2n-1,q)|$ or $|\Q^+(2n-1,q)|$ points;
\item[(c)] Every line through $N$ contains exactly one point of $\S$.
\end{itemize}
We see that $N\notin S$. Note that (a) follows immediately from (c). 

Consider a parabolic quasi-quadric as defined by us: this is a set of points $\S$ in $\PG(2n,q)$ such that
\begin{itemize}
\item[(b')] every hyperplane intersects $\S$ in $|\Q^-(2n-1,q)|$, $|\Q^+(2n-1,q)|$ or $|P\Q(2n-2,q)|$ points.
\end{itemize}

It is easy to see that a parabolic quasi-quadric with nucleus $N$ as defined in \cite{QQ} satisfies our definition of a parabolic quasi-quadric: 
\begin{lemma} \label{triviaal}A set $\S$ with nucleus $N$ satisfying (b)-(c) satisfies (b').
\end{lemma}
\begin{proof} It follows from $(c)$ that the number of points of $\S$ in a hyperplane through $N$ is the number of lines through a point in $\PG(2n-1,q)$, which is precisely $|P\Q(2n-2,q)|$. 
\end{proof}

\subsection{Properties of the nucleus}\label{sec:nucleus-properties}

The singular hyperplanes of a parabolic quadric are precisely those hyperplanes through the nucleus (see e.g. \cite[Lemma 1.13(ii)]{GGG}). Note that some authors use this as the definition for a parabolic quadric: they define the nucleus as the intersection point of all singular hyperplanes.

In order to investigate the situation for quasi-quadrics, we will consider the following properties, where it is clear that (d') implies (d).
\begin{itemize}
\item[(d)] Hyperplanes meeting $\S$ in $|P\Q(2n-2,q)|$ points contain a common point $N$;
\item[(d')] Hyperplanes meeting $\S$ in $|P\Q(2n-2,q)|$ points contain a common point $N$ and all hyperplanes through $N$ meet $\S$ in $|P\Q(2n-2,q)|$ points.
\end{itemize}

\begin{lemma} \label{nucleus2}\begin{itemize}
\item[(i)] Suppose that $\S$ and $N$ satisfy (a) and (b'), then the number of singular hyperplanes is $\frac{q^{2n}-1}{q-1}$.
\item[(ii)] Suppose that $\S$ and $N$ satisfy (a) and (b') and (d), then $\S$ satisfies (a)-(b)-(c) where $N$ is the nucleus.
\item[(iii)] Suppose that $\S$ and $N$ satisfy (a) and (b') and (c), then $\S$ satisfies (a)-(b)-(d').
\end{itemize}
\end{lemma}
\begin{proof}

\begin{itemize}\item[(i)]
Let $\alpha_1$ be the number of hyperplanes meeting $\S$ in $u_1=|\Q^-(2n-1,q)|$ points, $\alpha_2$ be the number of hyperplanes meeting $\S$ in $u_2=|\Q^+(2n-1,q)|$ points and $\alpha_3$ be the number of hyperplanes meeting $\S$ in $u_3=|P\Q(2n-2,q)|$ points (i.e. the number of singular hyperplanes). Standard counting yields that
\begin{align}
\alpha_1+\alpha_2+\alpha_3&=\frac{q^{2n+1}-1}{q-1}\\
\alpha_1 u_1+\alpha_2 u_2 +\alpha_3u_3&=\frac{q^{2n}-1}{q-1}\frac{q^{2n}-1}{q-1}\\
\alpha_1 u_1 (u_1-1)+\alpha_2u_2 (u_2-1)+\alpha_3 u_3(u_3-1)&=\frac{q^{2n}-1}{q-1}(\frac{q^{2n}-1}{q-1}-1)\frac{q^{2n-1}-1}{q-1}
\end{align}

Solving this system of equations, we find that $\alpha_3$, the number of singular hyperplanes is $\frac{q^{2n}-1}{q-1}$. 

\item[(ii)] Assume now that, on top of (a)-(b') also (d) holds, that is, that all singular hyperplanes contain a common point $N$. Since $\frac{q^{2n}-1}{q-1}$, the number of singular hyperplanes found in (i) is precisely the number of hyperplanes though $N$, (b') implies that all hyperplanes not through $N$ meet $\S$ in $|\Q^-(2n-1,q)|$ or $|\Q^+(2n-1,q)|$ points. Hence (b) follows. 

In order to show (c) we will show by induction that a codimension $j$-space through $N$ contains $\frac{q^{2n-j}-1}{q-1}$ points of $\S$. For $j=1$, this follows from the fact that $N$ is the intersection point of all singular hyperplanes, and singular hyperplanes contain $\frac{q^{2n-1}-1}{q-1}$ points of $\S$. Suppose the statement holds for all codimension $j_0$-spaces such that $1\leq j_0\leq 2n-2$ and consider a codimension $(j_0+1)$-space $\pi$. There are $\frac{q^{j_0+1}-1}{q-1}$ codimension $j_0$-spaces through $\pi$, each containing $\frac{q^{2n-j_0}-1}{q-1}-x$ points of $\S$ where $x$ is the number of points of $\S$ in $\pi$. Since $$(\frac{q^{2n-j_0}-1}{q-1}-x)(\frac{q^{j_0+1}-1}{q-1})+x=|S|=\frac{q^{2n}-1}{q-1}$$
we obtain that $x=\frac{q^{2n-(j_0+1)}-1}{q-1}$. The statement follows by induction, and taking $j_0=2n-1$.

\item[(iii)]
Now assume that (a),(b'),(c) are satisfied. By (c), all hyperplanes through $N$ are singular, and (i) shows that the number of singular hyperplanes is precisely the number of hyperplanes through $N$, so (b) follows. It also shows that the singular hyperplanes are precisely those hyperplanes going through $N$, so (d') follows.
\end{itemize}
\end{proof}

From Lemma \ref{nucleus2}(iii) we immediately obtain:

\begin{corollary}\label{nucleuswelldefined} For a parabolic quasi-quadric in $\PG(2n,q)$ of size $\Q(2n,q)$, $q$ even, the notions of nucleus as point only lying on tangent lines and as point lying on only singular hyperplanes coincide.
\end{corollary}

Consider now a set $\S$ satisfying (a) and (b') for $n=1$. This is a set of $q+1$ points in $\PG(2,q)$ such that every line meets it in $0,1$ or $2$ points, i.e. an oval. It is well-known (and an easy exercise) that for $q$ even, there is a unique point $N$ such that all tangent lines to the oval go through a common point $N$, the nucleus. In other words, when $n=1$ and $q$ is even, the properties (a) and (b') imply the existence of a point $N$ such that (c) holds. 

Proposition \ref{nonucleus} below shows that the case $n=1$ is exceptional: parabolic quasi-quadrics in $\PG(2n,q)$, $n>1$, do not necessarily have a nucleus. 

\subsection{Pivoting for parabolic quasi-quadrics} \label{sec:pivoting-PQQ}
In the following proposition, we show that we can pivot in a parabolic quadric and obtain a quasi-quadric without nucleus.
\begin{proposition} \label{nonucleus} 
A parabolic quasi-quadric in $\PG(2n,q)$, $q$ even, $n>1$, with $|\Q(2n,q)|$ points, does not necessarily have a nucleus.
\end{proposition}
\begin{proof}

Consider the parabolic quadric $\Q=\Q(2n,q)$ in $\PG(2n,q)$, $q$ even, $n\geq 2$. Then $\Q$ has a nucleus $N$. Let $\pi$ be a singular hyperplane of $\Q$, then $\pi$ contains $N$ and $\pi$ meets $\Q$ in the points of a cone $P\mathcal{C}$ where $\mathcal{C} = \Q(2n-2,q)$. Without loss of generality we can take the base $\mathcal{C}$ in an $(2n-2)$-space $\mu$ through $N$. We then see that the nucleus $N$ of $\Q$ is also the nucleus of $\mathcal{C}$. Now consider a parabolic quadric $\mathcal{C}'$ in $\mu$ with nucleus $N'\neq N$. We claim that the set $\Q'=(\Q\setminus P\mathcal{C})\cup P\mathcal{C}'$ obtained by pivoting in $\Q$ is a parabolic quasi-quadric without nucleus. The hyperplane $\pi$ is a singular hyperplane of $\Q'$ so it follows from Corollary \ref{nucleuswelldefined} that if $\Q'$ has a nucleus $N'$, then $N'$ is contained in $\pi$.

We will show that every point of $\pi$, not in $\P'$, lies on at least one line which is not a tangent line to $\P'$. First note that the point $N$ lies on a $2$-secant to $\mathcal{C}'$ by our construction, so $N$ does not only lie on tangent lines to $\P'$. Furthermore, a point of $\pi$ in $\Q$ which does not lie in $\P'$ is necessarily different from $P$, and hence, lies on lines of $\Q$ which are not contained in $\pi$. Since $\P$ and $\Q$ coincide outside $\pi$, these lines meet $\P'$ in $q\geq 2$ points, and hence, are not tangent lines.

It remains to show that a point of $\pi$, not in $\Q$, and different from $N$, lies on at least one $2$-secant to $\P'$. We will count the number of $2$-secants to $\Q$ and subtract the number of $2$-secants to $\Q$ in $\pi$. Since this is a positive integer and the points of $\Q$ not in $\pi$ are points of $\P'$, the statement then follows.

The stabiliser of $\Q$ in $\PG(2n,q)$ has three different point orbits: the first orbit consists of the points of $\Q$, the second of the nucleus, and the third of all other points (see e.g. \cite[Theorem 1.49]{GGG}). Thus the number $X$ of $2$-secants through a fixed point $\neq N$, of $\PG(2n,q)$, not in $\Q$, is a constant. 

The number of $2$-secants to $\Q$ is $\frac{|\Q|(|\Q|-|P\Q(2n-2,q)|)}{2}$ since $P$ is collinear with $|P\Q(2n-2,q)|$ points in $\Q$.

Now count couples $(R, L)$, where $R$ is a point, not on $\Q$ and different from $N$, and $L$ is a $2$-secant to $\Q$ through $R$.
$$(|\PG(2n,q)|-1-|\Q|)X=\frac{|\Q|(|\Q|-|P\Q(2n-2,q)|)}{2}.(q-1).$$

It follows that $X=\frac{q^{2n-1}}{2}$. 

To find the number of $2$-secants through a point $R$ different from $N$, not in $\Q$, of $\pi$ to the cone $P\Q(2n-2,q)$ in $\pi$, consider a hyperplane $\nu$ of $\pi$ through $R$ but not through $P$. The hyperplane $\nu$ meets $P\Q(2n-2,q)$ in a non-singular parabolic quadric $\tilde{\Q}$ in $\mu$ which implies by the previous count that $R$ lies on $\frac{q^{n-3}}{2}$ $2$-secants to $\tilde{Q}$. Now every $2$-secant $M$ through $R$ in $\pi$ gives rise to a unique $2$-secant $L$ through $R$ in $\nu$, namely $\langle P,M\rangle\cap \mu$. We see that every $2$-secant through $R$ in $\nu$ is determined $q$ times, namely, by every of the $q$ lines through $R$ in the plane $\langle L,P\rangle$, different from $RP$. It follows that the number of $2$-secants through $R$ in $\pi$ is $q\frac{q^{2n-3}}{2}=\frac{q^{2n-2}}{2}$. This means that there are $\frac{q^{2n-1}-q^{2n-2}}{2}>0$ $2$-secants to points of $\Q\setminus \pi$ through a point of $\pi$, different from $N$ and not in $\Q$.
\end{proof}

\begin{remark}\label{rem:pivotnucleus} In the definition of pivoting for parabolic quadrics with a nucleus given in \cite{QQ}, the authors replace a cone over a parabolic quadric with nucleus $N$ by a cone with the same vertex and base a different parabolic quadric with the same nucleus $N$. In this way they ensure that pivoting a parabolic quadric for $q$ even always gives rise to a parabolic quasi-quadric with nucleus $N$. 
In \cite[Lemma 1.6]{BHJS} it is proved that a parabolic quasi-quadric in $\PG(4,2)$ with nucleus is always a $\Q(4,2)$,
however, when we do not enforce the existence of a nucleus, we have seen that pivoting in $\Q(4,2)$ can result in a parabolic quasi-quadric in $\PG(4,2)$ which is not a quadric.
\end{remark}

We have seen in Corollary \ref{prop:Q4q} and Example \ref{ex:switchnotpivot}, that we can replace the pointset of a singular hyperplane of $\Q(2n,q)$, $q$ even, by a set of points which is not a cone over a quasi-quadric and obtain a parabolic quasi-quadric. But Remark \ref{rem:pivotnucleus} leads to the question: can we still do that if we want our obtained quasi-quadric to have a nucleus?
The following Corollary shows that this is not the case: when we want our quasi-quadric to have a nucleus, switching in singular hyperplanes is essentially pivoting, just like we have shown for all other polar spaces in Proposition \ref{prop:switchispivot}. The only difference is that in this case, we have two choices for the vertex of the cone to pivot with. 

\begin{corollary}\label{cor:parabolic}
Let $\P=\Q(2n,q)$, $n>1$, $q$ even and let $N$ be its nucleus. Suppose that $\P'$ is a parabolic quasi-polar space with nucleus $N'$ obtained by switching in a singular hyperplane $\pi$ of $\P$ with vertex $P$. Then either 
\begin{itemize}
\item $N=N'$ and $\P'\cap \pi$ is the point set of a cone with vertex $P$. If $\mu$ is a hyperplane of $\pi$ containing $N$ but not $P$ then the base of the cone is a parabolic quasi-quadric in $\mu$ with nucleus $N$. 
\item $N=P$ and $\P'\cap \pi$ is the point set of a cone with vertex $N$. If $\mu$ is a hyperplane of $\pi$ containing $P$ but not $N$ then the base of the cone is a parabolic quasi-quadric in $\mu$ with nucleus $P$.
\end{itemize} 

\end{corollary}
\begin{proof} Since $\pi$ is a singular hyperplane, $\pi$ contains the nucleus, say $N$, of $\P$. Since $\P'$ is a parabolic quasi-quadric with nucleus, say $N'$, we know that $|\P'|=|\P|$, and hence, that $\pi$ is a singular hyperplane for $\P'$ too, and hence,  $N'$ is contained in $\pi$. A point $Q$ of $\P\cap \pi$, except for $P$ lies on a line contained in $\P$ which is not contained in $\pi$. Since a nucleus only lies on $1$-secants to $\P'$, and $\P$ and $\P'$ coincide outside of $\pi$, $Q$ is not the nucleus of $\P'$. From Proposition \ref{nonucleus}, we know that every point $R$ of $\pi$, not in $\P$ and not $N$, lies on a $2$-secant to $\P$, not in $\pi$, so $R$ cannot be the nucleus of $\P'$. We conclude that $N=N'$. Since all hyperplanes of $\pi$ through $N'$ meet $\P$ in $|\Q(2n-2,q)|$ points, we have that all hyperplanes of $\pi$ not through $P$ meet $\P'$ in $|\Q(2n-2,q)|$ points, and hence, by Lemma \ref{lem:conelike}, $\P'$ is a cone with vertex $P$. Since all hyperplanes of $\pi$ through $P$ meet $\P'$ in $|P\Q^-(2n-3,q)|,|\Q(2n-2,q)|$ or $|P\Q^+(2n-3,q)|$ points by Proposition \ref{prop:conelike}, we see that $\P'$ has as base a parabolic quasi-quadric, say contained in a hyperplane $\mu$. The nucleus of this quasi-quadric is then the point $\mu\cap PN$.

Finally, if $N'=P$, then reversing the roles of $P$ and $N$ in the above reasoning yields that $\P'$ is the point set of a cone with vertex $N$ and base a parabolic quasi-quadric with nucleus on the line $PN$. 
\end{proof}

\subsection{A sufficient condition for the existence of a nucleus}\label{sec:nucleus-sufficient}
Keeping Proposition \ref{nonucleus} in mind, we see that in order to have the existence of a nucleus follow from our hypotheses on a quasi-quadric, we need to add an extra condition. Henceforth, we assume that $n\geq 2$ and we will modify condition $(c)$ into the following weaker one: 
\begin{itemize}
\item[(c')] Every codimension $2$-space is contained in at least one singular hyperplane.
\end{itemize}
We will show in Proposition \ref{equivalentdef} that the sets satisfying (a)-(b)-(c), that is, the parabolic quasi-quadrics with nucleus as defined in \cite{QQ}, are exactly those satisfying (a)-(b')-(c'). 

\begin{lemma} \label{easydirection} A parabolic quasi-quadric with nucleus $N$ (that is, a set of points satisfying (a)-(b)-(c)) satisfies (a)-(b')-(c'). 
\end{lemma}
\begin{proof} Let $\S$ be a set of points satisfying (a)-(b)-(c). From (b) we know that a hyperplane not through $N$ meets in $|\Q^+(2n-1,q)|$ or $|\Q^-(2n-1,q)|$ points, whereas (c) implies that the hyperplanes through $N$ meet in $|P\Q(2n-2,q)|$ points, hence $\S$ satisfies (b'). Now every codimension $2$ space $\Sigma$ either goes through $N$ and hence lies on only singular hyperplanes, or lies on exactly one singular hyperplane, namely $\langle \Sigma,N\rangle$.
\end{proof}

We will now show that the weaker condition (c') is strong enough to imply the existence of a nucleus.
\begin{lemma}\label{thereisnucleus} Suppose that $\S$ is a point set such that (a)-(b')-(c') hold, then (d') holds.
\end{lemma}
\begin{proof} We first show that every point is contained in $\frac{q^{2n-1}-1}{q-1}$ or $\frac{q^{2n}-1}{q-1}$ singular hyperplanes.
We count, in two ways, pairs $(P,\Pi)$ where $P$ is a point in a singular hyperplane $\Pi$ and then triples $(P,\Pi,\Pi')$ where $\Pi'$ is a hyperplane and $P$ is a point contained in the hyperplane intersection $\Pi\cap\Pi'$. Let $I$ index the points of $\PG(2n,q)$; then $|I| = \frac{(q^{2n+1}-1)}{q-1}$. Let $x_i$ be the number of singular hyperplanes through the $i$-th point. We find:

$$\sum_i x_i = \left(\frac{q^{2n}-1}{q-1}\right)^2$$
$$\sum_i x_i(x_i-1) = \frac{q^{2n}-1}{q-1} \left(\frac{q^{2n}-1}{q-1}-1\right)\frac{q^{2n-1}-1}{q-1}.$$
Since $|I| = \frac{q^{2n+1}-1}{q-1}$ we obtain
\begin{align}\sum_i \left(x_i-\frac{q^{2n-1}-1}{q-1}\right)\left(x_i-\frac{q^{2n}-1}{q-1}\right) = 0.\label{f1}\end{align}

We need to show that $x_i\geq \frac{q^{2n-1}-1}{q-1}$. This follows from an induction argument: each codimension $j$-space, where $j\geq 1$ lies on at least $\frac{q^j-1}{q-1}$ singular hyperplanes. 

The case $j=1$ is precisely (c'). Now let $\Sigma$ be a codimension $(j+1)$-space. Then $\Sigma$ is contained in $\frac{q^{2n-j}-1}{q-1}$
codimension $j$-spaces, each of which are contained in at least $\frac{q^j-1}{q-1}$ singular hyperplanes. The number of codimension $j$-spaces containing a codimension $(j+1)$-space in 
a hyperplane is equal to $\frac{q^{2n-j-1}-1}{q-1}$. 
Thus we obtain that $\Sigma$ is contained in at least $\frac{(q^{2n-j}-1)(q^j-1)}{(q^{2n-j-1}-1)(q-1)}>q\frac{q^j-1}{q-1}$ singular hyperplanes, which completes the inductive proof.

Using equation \eqref{f1}, it follows that, for all $i$, $x_i=\frac{q^{2n-1}-1}{q-1}$ or $x_i=\frac{q^{2n}-1}{q-1}$. Let $a_1$ be the number of points lying on $\frac{q^{2n-1}-1}{q-1}$ singular hyperplanes and $a_2$ be the number of points lying on $\frac{q^{2n}-1}{q-1}$ singular hyperplanes. Standard counting yields
$$a_1+a_2 = \frac{q^{2n+1}-1}{q-1}$$
$$a_1\frac{q^{2n-1}-1}{q-1}+a_2\frac{q^{2n}-1}{q-1}=(\frac{q^{2n}-1}{q-1})^2$$

It follows that $a_2=1$, so there exists a unique point, say $N$ lying on all $\frac{q^{2n}-1}{q-1}$ singular hyperplanes.
\end{proof}

We now show that we can deduce the size of a parabolic quasi-quadric if we impose condition (d').
\begin{lemma}\label{sizeparabolic} If (b') and (d') hold, then (a) holds.\end{lemma}
\begin{proof} As in Lemma \ref{nucleus2}, let $\alpha_1$ be the number of hyperplanes meeting $\S$ in $u_1=|\Q^-(2n-1,q)|$ points, $\alpha_2$ be the number of hyperplanes meeting $\S$ in $u_2=|\Q^+(2n-1,q)|$ points.  By (d') we know that the number of hyperplanes meeting $\S$ in $u_3=|P\Q(2n-2,q)|$ is $\frac{q^{2n}-1}{q-1}$ and we find:
\begin{align}
\alpha_1+\alpha_2&=\frac{q^{2n+1}-1}{q-1}-\frac{q^{2n}-1}{q-1}\label{e1}\\
\alpha_1 u_1+\alpha_2 u_2&=|S|-\frac{q^{2n}-1}{q-1}u_3\label{e2}\\
\alpha_1 u_1 (u_1-1)+\alpha_2u_2 (u_2-1)&=|S|(|S|-1)\frac{q^{2n-1}-1}{q-1}-\frac{q^{2n}-1}{q-1}u_3(u_3-1)\label{e3}\end{align}

Using the first two equations to write $\alpha_1$ and $\alpha_2$ in function of $|S|$, we see that the third yields a quadratic equation in $|S|$ whose sum of roots is given by
$\frac{(q^{2n}-1)(u_1+u_2-1)}{q^{2n-1}-1}+1$. Since $|S|=|\Q(2n,q)|=\frac{q^{2n}-1}{q-1}$ is a solution, the other solution is an integer if and only if
$$\frac{(q^{2n}-1)(u_1+u_2-1)}{q^{2n-1}-1}$$ is an integer.

This expression equals $$ \frac{q^{2n}-1}{q^{2n-1}-1}(2\frac{q^{2n-1}-1}{q-1}-1)=2\frac{q^{2n}-1}{q-1}-\frac{q^{2n}-1}{q^{2n-1}-1}$$ and we see that the latter term in not an integer for $n\geq 2$.
\end{proof}

Finally, we show that when a parabolic quasi-quadric has a nucleus, then $q$ is necessarily even. 
\begin{lemma}\label{lem:nucleus-even} If (b') and (d') hold for a point set in $\PG(2n,q), n\geq 2$, then $q$ is even.
\end{lemma}
\begin{proof}
Suppose that (b') and (d') hold, then Lemma \ref{sizeparabolic} implies that $|\S|=|\Q(2n,q)|$. The singular hyperplanes are precisely the hyperplanes containing $N$. Note that if $q$ is odd, $|\S|=|\Q(2n,q)|$ is even, $|\Q^\pm(2n-1,q)|$ is even and $|P\Q(2n-2,q)|$ is odd. Hence, when we consider the set $\S'=\S\cup \{N\}$ we see that every hyperplane meets $\S'$ in an even number of points. Note that $|\S'|=1\pmod{2}$.
We claim that a codimension $j$ space has $j+1\pmod 2$ points of $\S'$ and will proceed to show this by induction. We have just established the base case $j=1$ so suppose our claim holds for some $1\leq j_0\leq 2n-2$ and consider a codimension $j_0+1$ space $\pi$. There are $\frac{q^{j_0}-1}{q-1}$ codimension $j_0$-spaces through $\pi$ which is a number congruent to $j_0\pmod 2$.  Suppose to the contrary that $\pi$ contains $j_0\pmod 2$ points, then, by our induction hypothesis, each codimension $j_0$ space through $\pi$ contains $1\pmod 2$ points of $\S$ outside $\pi$, and we find that $|\S'|-x=j_0\pmod 2$. Since $x=j_0\pmod 2$, we find that $|\S'|=0 \mod 2$, a contradiction. By induction, our claim holds. In particular, we see that every line, which is a codimension $2n-1$ space, contains an even number of points of $\S'$. Now consider a hyperplane $\mu$ meeting $\S$ in $|\Q^-(2n-1,q)|$ points and let $P$ be a point of $\S$. Since every line through $\S$ needs to contain at least one extra point of $\S\cup \{N\}$, we find that there are at least $\frac{q^{2n-1}-1}{q-1}$ points of $\S$ in $\mu$, a contradiction since $\frac{q^{2n-1}-1}{q-1}>|\Q^-(2n-1,q)|$.
\end{proof}

Combining the above Lemma's, we see that the parabolic quasi-quadrics as defined by the properties (a)-(b)-(c) are precisely the parabolic quasi-quadrics satisfying the weaker hypotheses (a)-(b')-(c').
\begin{proposition} \label{equivalentdef}The parabolic quasi-quadrics with nucleus as defined in \cite{QQ} are precisely those point sets $\S$ in $\PG(2n,q)$, satisfying
(a)-(b')-(c').
\end{proposition}
\begin{proof} If $n=1$, it is well known that an oval in $\PG(2,q)$, $q$ even has a nucleus. It is clear that the point sets satisfying (a)-(b)-(c), as well as those satisfying (a)-(b')-(c') are precisely the ovals. If $n>1$, then we have seen in Lemma \ref{easydirection} that every parabolic quasi-quadric with nucleus as defined in \cite{QQ} satisfies (a)-(b')-(c'). On the other hand, if $\S$ is a set of points satisfying (a)-(b')-(c'), we have shown in Lemma \ref{thereisnucleus} that (d') (and hence (d)) holds. Lemma \ref{nucleus2} then shows that $\S$ indeed satisfies (a)-(b)-(c). Finally, Lemma \ref{lem:nucleus-even} shows that $q$ is indeed even.
\end{proof}

\begin{corollary} The parabolic quasi-quadrics that have the property that the singular hyperplanes are precisely those through a common point (i.e. they satisfy (d')) are precisely the parabolic quasi-quadrics with nucleus as defined in \cite{QQ}.
\end{corollary}
\begin{proof} For $n=1$, this follows from the fact that a parabolic-quadric that satisfies property (d') necessarily has size $q+1$, and hence, is an oval. For $n\geq 2$, we have seen in Lemma \ref{sizeparabolic} that every parabolic quasi-quadric satisfying (d') has size $|\Q(2n,q)|$, so (a) holds. Lemma \ref{nucleus2}(ii) then shows that $\S$ indeed satisfies (a)-(b)-(c) and  Lemma \ref{lem:nucleus-even} shows that $q$ is indeed even.
\end{proof}

\begin{remark} \label{rem:sizeparabolic} We have shown in this paper that a point set satisfying (b') and (c') has size $|\Q(2n,q)|$. It is an interesting open question to see whether (b') implies (a). It is clear that for $n=1$, any arc will satisfy (b'), and hence, (a) does not follow from (b') when $n=1$. However, we were unable to construct an example of a parabolic quasi-quadric in $\PG(2n,q)$, $n>1$ that does not satisfy (a). 

Using similar standard counting arguments as in Lemma \ref{nucleus2}, one can show that the size of a parabolic quasi-quadric in $\PG(4,q)$ is congruent to $1$ mod $q$ and lies in the interval $[q^3+q^2-q\sqrt{q}+1,q^3+q^2+q\sqrt{q}+1]$. There are several values in this interval for which the corresponding solutions to the system of Lemma \ref{nucleus2} are positive integers, which explains our inability to deduce that the size would necessarily be $q^3+q^2+q+1$.
Furthermore, showing that the size of a quasi-quadric is $|\Q(2n,q)|$ under the additional hypothesis of knowing the intersection sizes with co-dimension $2$-spaces, is already fairly complicated (see \cite[Lemma 2.21]{SDW-JS2}). 

\end{remark}

\section{Appendix A}\label{AppendixA}

{\bf Proof of Lemma \ref{lem:switching-nonsingular} cont.:}

\begin{proof}
{\bf Case 2: $\P = \Q(2n,q)$}, $n\geq 2$. We know that every hyperplane meets $\P$ in either $|\Q^-(2n-1,q)|$, $|P\Q(2n-2,q)|$ or $|\Q^+(2n-1,q)|$ points. Suppose that $\pi$ is a non-singular hyperplane (i.e., meeting $\P$ in a $\Q^\pm(2n-1,q)$), and that $\P'$ is a quasi-quadric obtained by switching in $\pi$. We will show that $\P=\P'$.

If $\mu$ is a hyperplane of $\pi$ meeting $\Q^\pm(2n-1,q)$ in $|P\Q^\pm(2n-3,q)|$ points, then $\mu$ is contained in exactly $1$ singular hyperplane and $q$ non-singular hyperplanes. So the hyperplanes through $\mu$ have either  $|P\Q(2n-2,q)|-|P\Q^\pm(2n-3,q)|$ or $|\Q^\pm(2n-1,q)|-|P\Q^\pm(2n-3,q)|$ points of $\P$ outside $\pi$. This implies that $\mu\cap \P'$ has either $|\Q(2n-2,q)|$ or $|P\Q^\pm(2n-3,q)|$ points.

If $\mu$ is a hyperplane of $\pi$ meeting $\Q^\pm(2n-1,q)$ in $|\Q(2n-2,q)|$ points, then $\mu$ lies on at most $2$ singular hyperplanes, and on at least $(q-1)/2$ hyperplanes of elliptic and hyperbolic type. Hence, if $q\geq 4$, there is at least one elliptic and one hyperbolic hyperplane through $\mu$, different from $\pi$. It follows that $|\mu\cap \P'|=|\Q(2n-2,q)|$.

Hence, we find that $\P'\cap \pi$ is a set $\T$ of points in $\PG(2n-1,q)$ such that every hyperplane meets $\T$ in either $|\Q(2n-2,q)|$ or $|P\Q^\pm(2n-3,q)|$ points. Thus $\P'\cap \pi$ is a quasi-quadric in $\PG(2n-1,q)$, and has, since $q>2$ by Lemma \ref{lem:correct-points}, $|\Q^\pm(2n-1,q)|$ points. Moreover, we see that hyperplanes meeting $\P\cap \pi$ in $|\Q(2n-2,q)|$ points meet $\T$ in $|\Q(2n-2,q)|$ points. There are only two types of hyperplanes, and the number of each of them is determined (see Lemma \ref{lem:correct-points}). It follows that hyperplanes of $\pi$ meeting $\P$ in $|P\Q^\pm(2n-3,q)|$ points meet $\P'$ in $|P\Q^\pm(2n-3,q)|$ points. By Lemma \ref{lem:equal-pointsets}, the point sets $\P\cap \pi$ and $\P'\cap \pi$ coincide. This implies that $\P=\P'$.

{\bf Case 3: $\P=\H(m,q^2)$}. Let $\P=\H$ be a Hermitian variety $\H(m,q^2)$. We know that every hyperplane meets $\H$ in either $|\H(m-1,q^2)|$ or $|P\H(m-2,q^2)|$ points. Suppose that $\pi$ is a non-singular hyperplane (i.e., meeting $\H$ in a $\H(m-1,q^2)$) and that $\H'$ is a quasi-Hermitian variety obtained by switching in $\pi$. We will show that $\H=\H'$.
By Corollary \ref{samesize} we know that $|\H|=|\H'|$.

If $\mu$ is a hyperplane of $\pi$, then $\mu$ lies on both singular and non-singular hyperplanes, different from $\pi$. This implies that $\mu$ meets $\H$ in the same number of points as $\H'$. Using Lemma \ref{lem:equal-pointsets}  we find that  $\H\cap\pi$ and $\H'\cap\pi$ coincide. Hence, $\H=\H'$.
\end{proof}

\begin{proposition}\label{Q23}
Let $\P$ be an oval in $\PG(2,3)$ and let $\P'$ be an oval obtained by switching $\P$ in the line $L$. Then $\P=\P'$ or $L$ is a secant line to $\P$ and $\P'$ is obtained by removing one of the two intersection points of $\P$ with $L$ and adding the unique external point to the oval on the line $L$.\end{proposition} 
\begin{proof} Since both $\P$ and $\P'$ are ovals, and $\P'$ is obtained by switching $\P$, we know that $|\P\cap L|=|\P'\cap L|$. If $L$ is a tangent line to $\P$ in the point $P$, then we see that every point of $L$, different from $P$ lies on a secant line and a passant to $\P$. It follows that for a point $P$ on $L$, $P\notin \P'$ implies $P\notin \P'$. Since $|\P\cap L|=|\P'\cap L|$ and $\P$ and $\P'$ coincide outside $L$, we have that $\P=\P'$.
Now let $L$ be a secant line to $\P$, let $P_1,P_2$ be the intersection points of $L$ with $\P$ and let $P_3,P_4$ be the points of $\P$, not on $L$. The point $Q=P_3P_4\cap L$ does not lie in $\P$, lies on a $2$-secant and a passant to $\P$. It follows that $Q\notin \P'$. The point $R$ on $L$, different from $P_1,P_2$ and $Q$ is the unique external point to $\P$ on $L$ and lies on two tangent lines and one passant. It follows that if $\P\neq \P'$, then the point $R$ was added and one of $P_1,P_2$ was removed; note that the set $\{P_3,P_4,P_i,R\}$, $i=1,2$ is indeed a set of $4$ points, no three of which lie on a line, i.e. an oval.
\end{proof}

\section{Appendix B}\label{AppendixB}

We conclude by discussing switching in $\Q(2n,2)$ and $\Q(2n,3)$.

\begin{proposition}\label{Q2n2}
There exists a quasi-quadric, which is not a quadric, obtained by switching in non-singular hyperplane of the parabolic quadric $\Q(2n,2)$, $n\geq 2$.
\end{proposition}
\begin{proof} Let $\P=\Q(2n,2)$, let $N$ be its nucleus and let $\pi$ be a non-singular hyperplane meeting $\Q(2n,2)$ in $\Q^\pm(2n-1,2)$. Since $\pi$ is non-singular, $N\notin \pi$. Let $\P'$ be the set $\P\setminus (\P\cap\pi)\cup \Q'$ where $\Q'$ is a quasi-quadric in $\pi=\PG(2n-1,2)$ of the same type as $\Q^\pm(2n-1,2)$. 
We will show that $\P'$ is a parabolic quasi-quadric. Consider a hyperplane $H$ of $\PG(2n,2)$. If $H=\pi$ then $|H\cap \P|=|H\cap \P'|=|\Q^\pm(2n-1,2)|$ and we are done. If $H\neq \pi$ then $H$ meets $\pi$ in a hyperplane $\mu$ of $\pi$. Then $\mu$ lies on exactly one singular hyperplane (the hyperplane through $\mu$ and $N$). This shows that the other, non-singular, hyperplane must be of the same type as $\P\cap\pi$ if $\mu\cap\P=P\Q^\pm(2n-3,2)$ and of the opposite type of $\mu\cap \P=\Q(2n-2,2)$.

If $\mu\cap \P=\Q(2n-2,2)$, then the singular hyperplane contains $|P\Q(2n-2,2)|-|\Q(2n-2,2)|=2^{2n-2}$ points of $\P$ outside $\pi$ and the non-singular hyperplane contains the other $|\P-(\P\cap \pi)|-2^{2n-2}=|\Q(2n,2)|-|\Q^\pm(2n-1,2)|-2^{2n-2}=2^{2n-2}\mp2^{n-1}$ points outside $\pi$.
If $|\mu\cap \pi|=|P\Q^\pm(2n-3,2)|$, then the singular hyperplane contains  $|P\Q(2n-2,2)|-|P\Q^\pm(2n-3,2)|=2^{2n-2}\mp2^{n-1}$ points outside $\pi$ and the non-singular hyperplane contains the other $|\P-(\P\cap \pi)|-(2^{2n-2}\mp 2^{n-1})=|\Q(2n,2)|-|\Q^\pm(2n-1,2)|\mp (2^{2n-2}-2^{n-1})=2^{2n-2}$ points. Hence, all hyperplanes, different from $\pi$ have $2^{2n-2}\mp2^{n-1}$ or $2^{2n-2}$ points of $\P$ outside $\pi$.
Since $\Q'$ is a quasi-quadric, every hyperplane of $\pi$ meets $\P'\cap\pi=\Q'$ in $|\Q(2n-2,2)|$ or $|P\Q^\pm(2n-3,2)|$ points.
Note that $2^{2n-2}\mp2^{n-1}+|\Q(2n-2,2)|=|\Q^\mp(2n-1,2)|$, $2^{2n-2}+|\Q(2n-2,2)|=|P\Q(2n-2,2)|$, $2^{2n-2}\mp 2^{n-1}+|P\Q^\pm(2n-3,2)|=|P\Q(2n-2,2)|$ and $2^{2n-2}+|P\Q^\pm(2n-3,2)|=|\Q^\pm(2n-1,2)|$.
It follows that every hyperplane meets $\P'$ in $|\Q^-(2n-1,2)|,|P\Q(2n-2,2)|$ or $\Q^+(2n-1,2)|$ points.
Finally, we need to show that it is possible to pick $\Q'$ such that $\P'$ is not a parabolic quadric. For $n\geq 2$, let $\Q'$ be a quasi-quadric in $\pi$, which is not a quadric (see Remark \ref{remark36}), then $\P'$ is not a quadric since a parabolic quadric cannot contain a subspace $\pi$ with $|\Q^\pm(2n-1,2)|$ points that do not form a $\Q^\pm(2n-1,2)$. For $\Q(4,2)$, we claim that there is always a quadric $\Q'$ such that $\P'$ does not have a nucleus. This can be readily checked by the use of a computer, e.g. using the package FinInG \cite{fining}: of the 280 choices for a hyperbolic quadric $\Q^+(3,2)$, 270 give rise to a quasi-quadric without nucleus, of the 168 choices for an elliptic quadric, 162 give rise to a quasi-quadric without nucleus. \end{proof}

\begin{proposition}\label{Q2n3}
There exists a quasi-quadric, which is not a quadric, obtained by switching in a non-singular hyperplane of the parabolic quadric $\Q(2n,3)$, $n\geq 2$.
\end{proposition}

\begin{proof}
Let $\xi$ be a hyperplane intersecting $\P=\Q(2n,3)$ in a non-singular quadric $\Q^\pm(2n-1,3)$  and let $\pi$ be a singular hyperplane of $\xi$ to $\Q^\pm(2n-1,3)$. Then $\pi\cap \P$ is a cone $P\tilde{\Q}$ with vertex a point $P$ and base a non-singular quadric $\Q^\pm(2n-3,3)$. Let $\P'$ be the set of points of $(\P\setminus (\P\cap \xi)) \cup \S$, where $\S$ is the set of internal points to $\P$ in $\xi$, not contained in $\pi$, together with the points of $P\tilde{\Q}$. In other words, $\P'$ is obtained by taking the points of $\P$ and replacing the points of $\P$ in $\xi\setminus\pi$ by the internal points to $\P$ contained in $\xi\setminus \pi$.

The number of internal points to $\P$ in $\PG(2n,q)$, $q$ odd, is the number of elliptic hyperplanes, which is $\frac{1}{2}q^n(q^n-1)$, and the number of internal points in $\xi$ is the number of hyperplanes of the type of $\xi$ through an internal point, which is $\frac{1}{2}q^{n-1}(q^n\mp1)$ (see e.g. \cite[Sections 1.5--1.7]{GGG}).

We also see that the number of internal points in a hyperbolic hyperplane is the number of elliptic hyperplanes through an external point, which is $\frac{1}{2}q^{n-1}(q^n-1)$. 

We now determine the number of internal points in codimension $2$-spaces.
\begin{itemize}
\item A codimension $2$-space $\mu$ for which $\mu^\perp$ is a passant lies on $2$ elliptic and $2$ hyperbolic hyperplanes, from which it follows that $\mu$ has $\frac{1}{2}3^{n-1}(3^n+1)-3^{2n-2}$ internal points. 
\item A codimension $2$-space $\mu$ for which $\mu^\perp$ is a $2$-secant lies on $1$ elliptic, $1$ hyperbolic and $2$ singular hyperplanes, from which is follows that $\mu$ has $\frac{1}{2}3^{n-1}(3^n-1)-3^{2n-2}$ internal points. 
\item Finally, a codimension $2$-space $\mu$ for which $\mu^\perp$ is a tangent line either meets $\P$ in a cone over an elliptic quadric, in which case $\mu$ lies on one cone and $3$ elliptic hyperplanes, which implies that $\mu$ has $\frac{1}{2}3^{n-1}(3^n+1)-3^{2n-2}$ internal points, or meets $\P$ in a cone over a hyperbolic quadric, in which case $\mu$ lies on one cone and $3$ hyperbolic hyperplanes, which implies that $\mu$ has $\frac{1}{2}3^{n-1}(3^n-1)-3^{2n-2}$ internal points.
\end{itemize}

In particular, it follows that there are $\frac{1}{2}3^{n-1}(3^n\mp1)-(\frac{1}{2}3^{n-1}(3^n\mp1)-3^{2n-2})=3^{2n-2}$  internal points in $\xi$, not in $\pi$. 

This shows that $|\S|=|P\Q^\pm(2n-3,3)|+3^{2n-2}=|\P\cap \xi|=|\Q^\pm(2n-1,3)|$.

Every hyperplane of $\PG(2n,3)$, different from $\xi$, intersects $\xi$ in a hyperplane of $\xi$, which either meets $\P$ in $|P\Q^\pm(2n-3,3)|$ or $|\Q(2n-2,3)|$ points. 

A hyperplane $\mu$ of $\xi$ with $|P\Q^\pm(2n-3,3)|$  points of $\P$ is such that $\mu^\perp$ is a tangent line to $\P$. Then $\mu$ lies only in hyperplanes meeting $\P$ in $|P\Q(2n-2,3)|$ or $|\Q^\pm(2n-1,3)|$ points, and as such, we need to show that $\mu$ has either $|P\Q^\pm(2n-3,3)|$  or $|P\Q^\pm(2n-3,3)|\mp 3^{n-1}=|\Q(2n-2,3)|$  points of $\P'$ in order for these hyperplanes to have $|\Q^-(2n-1,3)|$, $|P\Q(2n-2,3)|$ or $|\Q^+(2n-1,3)|$ points of $\P'$.

A hyperplane $\mu$ of $\xi$ with $|\Q(2n-2,3)|$ points of $\P$ for which $\mu^\perp$ is a $2$-secant to $\P$ lies on $\xi$, on $2$ hyperplanes meeting $\P$ in $|P\Q(2n-2,3)|$ and one hyperplane meeting in  $|\Q^\mp(2n-1,3)|$ points. Hence, we need to show that $\mu$ has either $|\Q(2n-2,3)|$ or $|\Q(2n-2,3)|\pm 3^{n-1}=|P\Q^\pm(2n-3,3)|$ points of $\P'$ in order for these hyperplanes to have $|\Q^-(2n-1,3)|$, $|P\Q(2n-2,3)|$ or $|\Q^+(2n-1,3)|$ points.

Finally, a hyperplane $\mu$ of $\xi$ with $|\Q(2n-2,3)|$ points of $\P$ for which $\mu^\perp$ is a passant to $\P$ lies on $\xi$, on $2$ hyperplanes meeting $\P$ in $|\Q^\mp(2n-1,3)|$ points and on one hyperplane, different from $\xi$ meeting $\P$ in $|\Q^\pm(2n-1,3)|$ points. Hence, we need to show that $\mu$ has exacty $|\Q(2n-2,3)|$ points of $\P'$ in order for these hyperplanes to have $|\Q^-(2n-1,3)|$ of $|\Q^+(2n-1,3)|$ points of $\P'$.

Now let $\nu$ be a hyperplane. If $\nu=\pi$ then we have argued above that $|\nu\cap \P|=|\nu\cap \P'|$ and we are done. So let $\nu\neq \pi$ and  let $L$ be the intersection of $\nu$ with $\pi$, which is a hyperplane of $\xi$.

There are $3$ possibilities for $L\cap \P$ if $\xi$ is an elliptic hyperplane.
\begin{itemize}
\item $L\cap \P$ is an elliptic quadric $\Q^-(2n-3,3)$. There are $4$ hyperplanes of $\xi$ through $L$, one of which is $\pi$. Now $L^\perp \cap \P$ is a conic, $\pi^\perp$ a tangent line to $\P$, contained in $L^\perp$, and $\xi^\perp$ is an internal point in $L^\perp$ on $\pi^\perp$. Hence, through $\xi^\perp$, there are two tangent lines to $L^\perp\cap \P$ in $L^\perp$, one secant line, and one external line.
It follows that there are $3$ hyperplanes of $\xi$ through $L$ with $\frac{1}{2}3^{n-1}(3^n+1)-3^{2n-2}$ internal points and one hyperplane (corresponding to the unique $2$-secant to $\P$ through $\xi^\perp$ in $L^\perp$) with $\frac{1}{2}3^{n-1}(3^n-1)-3^{2n-2}$ points. It follows that $L$ contains $\frac{1}{2}3^{n-2}(3^{n-1}+1)$   internal points.

Hence, if $\nu^\perp$ is a tangent or external line to $\P$, we find $|\nu^\perp\cap \P' |=\frac{1}{2}3^{n-1}(3^n+1)-3^{2n-2}-\frac{1}{2}3^{n-2}(3^{n-1}+1)+|\Q^-(2n-3,3)|=|\Q(2n-2,3)|$ points as needed.
And if $\nu^\perp$ is a $2$-secant line to $\P$, we find $|\nu^\perp\cap \P' |=\frac{1}{2}3^{n-1}(3^n-1)-3^{2n-2}-\frac{1}{2}3^{n-2}(3^{n-1}+1)+|\Q^-(2n-3,3)|=|P\Q^-(2n-3,3)|$ as needed.

\item $L\cap \P$ is a cone $M\Q^-(2n-5,3)$, where $M$ is a line through $P$. There are $4$ hyperplanes of $\xi$ through $L$, one of which is $\mu$. Now $L^\perp \cap \P$ is a line in the plane $L^\perp$, and $\xi^\perp$ is a point in $L^\perp$, not in $\P$. Hence, through $\xi^\perp$, there are four tangent lines to $L^\perp\cap \P$.
It follows that all $4$ hyperplanes of $\xi$ through $L$ have $\frac{1}{2}3^{n-1}(3^n+1)-3^{2n-2}$ internal points and it then follows that $L$ contains $\frac{1}{2}3^{n-1}(3^n+1)-3^{2n-2}-3^{2n-3}$ internal points. Hence, any hyperplane $\nu$ of $\xi$ through $L$ has $\nu^\perp$ a tangent line to $\P$, and we find $|\nu^\perp\cap \P' |=\frac{1}{2}3^{n-1}(3^n+1)-3^{2n-2}-(\frac{1}{2}3^{n-1}(3^n+1)-3^{2n-2}-3^{2n-3})+|M\Q^-(2n-5,3)|=|P\Q^-(2n-3,3)|$ as needed.

\item $L\cap \P$ is a cone $P\Q(2n-4,3)$. There are $4$ hyperplanes of $\xi$ through $L$, one of which is $\mu$. Now $L^\perp \cap \P$ are two lines, and $\xi^\perp$ is an internal point in $L^\perp$. Through $\xi^\perp$, there is one tangent line to $L^\perp\cap \P$, and $3$ $2$-secants.
It follows that the $3$ hyperplanes of $\xi$, different from $\mu$, through $L$ have $\frac{1}{2}3^{n-1}(3^n-1)-3^{2n-2}$ internal points and $\mu$ has $\frac{1}{2}3^{n-1}(3^n+1)-3^{2n-2}$ points. It follows that $L$ contains $\frac{1}{2}3^{n-1}(3^{n-1}-1)-3^{2n-3}$ internal points. Hence, for any hyperplane $\nu\neq \pi$ in $\xi$ through $L$, $\nu^\perp$ is a $2$-secant line to $\P$. Hence, we find $|\nu^\perp\cap \P' |=\frac{1}{2}3^{n-1}(3^n-1)-3^{2n-2}-\frac{1}{2}3^{n-1}(3^{n-1}-1)-3^{2n-3}+|P\Q(2n-4,3)|=|\Q(2n-2,3)|$ as needed.
\end{itemize}

There are $3$ possibilities for $L\cap \P$ if $\xi$ is a hyperbolic hyperplane:
\begin{itemize}
\item $L\cap \P$ is a hyperbolic quadric $|\Q^+(2n-3,3)|$. There are $4$ hyperplanes of $\xi$ through $L$, one of which is $\pi$. Now $L^\perp \cap \P$ is a conic, $\pi^\perp$ a tangent line to $\P$, contained in $L^\perp$, and $\xi^\perp$ is an external point in $L^\perp$ on $\pi^\perp$. Hence, through $\xi^\perp$, there are two tangent lines to $L^\perp\cap \P$ in $L^\perp$, one secant line, and one external line.

It follows that there are $3$ hyperplanes of $\xi$ through $L$ with $\frac{1}{2}3^{n-1}(3^n-1)-3^{2n-2}$ internal points and one hyperplane (corresponding to the unique passant to $\P$ through $\xi^\perp$ in $L^\perp$) with $\frac{1}{2}3^{n-1}(3^n+1)-3^{2n-2}$. It follows that $L$ contains $\frac{1}{2}3^{n-2}(3^{n-1}-1)$ internal points.

Hence, if $\nu^\perp$ is a tangent or secant line to $\P$, we find $|\nu^\perp\cap \P' |=\frac{1}{2}3^{n-1}(3^n-1)-3^{2n-2}-\frac{1}{2}3^{n-2}(3^{n-1}-1)+|\Q^+(2n-3,3)|=|\Q(2n-2,3)|$  points as needed.

And if $\nu^\perp$ is a passant to $\P$, we find $|\nu^\perp\cap \P' |=\frac{1}{2}3^{n-1}(3^n+1)-3^{2n-2}-\frac{1}{2}3^{n-2}(3^{n-1}-1)+|\Q^+(2n-3,3)|=|P\Q^+(2n-3,3)|$ as needed.

\item $L\cap \P$ is a cone $M\Q^+(2n-5,3)$ where $M$ is a line through $P$. There are $4$ hyperplanes of $\xi$ through $L$, one of which is $\mu$. Now $L^\perp \cap \P$ is a line in the plane $L^\perp$, and $\xi^\perp$ is a point in $L^\perp$, not in $\P$. Hence, through $\xi^\perp$, there are four tangent lines to $L^\perp\cap \P$ in $L^\perp$.
It follows that all $4$ hyperplanes of $\xi$ through $L$ have $\frac{1}{2}3^{n-1}(3^n-1)-3^{2n-2}$ internal points and it then follows that $L$ contains $\frac{1}{2}3^{n-1}(3^{n-2}-1)$ internal points.

Hence, any hyperplane $\nu$ of $\xi$ through $L$ has $\nu^\perp$ a tangent line to $\P$, and we find $|\nu^\perp\cap \P' |=\frac{1}{2}3^{n-1}(3^n-1)-3^{2n-2}-\frac{1}{2}3^{n-1}(3^{n-2}-1)
+|M\Q^+(2n-5,3)|=|P\Q^+(2n-3,3)|$ as needed.

\item $L\cap \P$ is a cone $P\Q(2n-4,3)$. There are $4$ hyperplanes of $\xi$ through $L$, one of which is $\pi$. Now $L^\perp \cap \P$ are two lines, and $\xi^\perp$ is a point in $L^\perp$, not in $\P$. Through $\xi^\perp$, there is one tangent line to $L^\perp\cap \P$, corresponding to $\pi$ and $3$ $2$-secants.
It follows that all hyperplanes of $\xi$ through $L$ have $\frac{1}{2}3^{n-1}(3^n-1)-3^{2n-2}$ internal points, from which it follows that $L$ contains $\frac{1}{2}3^{n-1}(3^{n-2}-1)$ internal points.

Hence, for any hyperplane $\nu\neq \pi$ in $\xi$ through $L$, $\nu^\perp$ is a $2$-secant line to $\P$. Hence, we find $|\nu^\perp\cap \P' |=\frac{1}{2}3^{n-1}(3^n-1)-3^{2n-2}-\frac{1}{2}3^{n-1}(3^{n-2}-1)+|P\Q(2n-4,3)|=|\Q(2n-2,3)|$ as needed.

It remains to show that $\P'$ is not a quadric. Consider a line $M$ which is a $2$-secant to $\P$ and intersects $\xi$ in a point of $\P'\setminus \P$. Then $M$ contains $3$ points of $\P'$. Since lines intersect quadrics in $0,1,2$ or $4$ points we conclude that  $\P'$ is not a quadric.
\end{itemize}
\end{proof}

{\bf Proof of Proposition \ref{prop:switching-is-pivoting}}:

\begin{proof} 
There are three intersection numbers of hyperplanes with respect to $\P$. Let $2n=m$, $A_{m-1}=|\Q^-(m-1,q)|$, $B_{m-1}=|P\Q(m-2,q)|$, $C_{m-1}:=|\Q^+(2n-1,q)$. Now let $H$ be a hyperplane of $\PG(m,q)$. We need to show that $H\cap \P'\in \{A_{m-1},B_{m-1},C_{m-1}\}$. Recall that $\P\cap (H\setminus \pi)=\P'\cap (H\setminus \pi)$. 

If $H$ does not contain $P$, then $H\cap \pi$ has the same number of points of both $\P$ and $\P'$, so $|H\cap \P'|=|H\cap \P|\in \{A_{m-1},B_{m-1},C_{m-1}\}$.
So suppose that $H$ contains $P$, then, since every hyperplane of $\mu$ meets the base $\tilde{\P}$ in either $A_{m-3}$, $B_{m-3}$ or $C_{m-3}$ points, $H\cap \pi$ meets in $\P$ in $qA_{m-3}+1$, $qB_{m-3}+1$, or $qC_{m-3}+1$ points. Likewise, $H\cap \pi$ meets $\P'$ in $qA_{m-3}+1$, $qB_{m-3}+1$, or $qC_{m-3}+1$ points.
If $|H\cap \pi\cap\P|=qA_{m-3}+1$, then an easy count shows that all hyperplanes through $\nu$, different from $\pi$, have $A_{m-1}$ points of $\P$.
We find that $|H\cap \P'|$ is either $A_{m-1}$, $A_{m-1}-(qA_{m-3}+1)+(qB_{m-3}+1)$, or  $A_{m-1}-(qA_{m-3}+1)+(qC_{m-3}+1)$. Since $qB_{m-3}-qA_{m-3}=B_{m-1}-A_{m-1}$ and $qC_{m-3}-qA_{m-3}=C_{m-1}-A_{m-1}$ we indeed have that $|H\cap \P'|\in \{A_{m-1},B_{m-1},C_{m-1}\}$. 
Similarly, if $|H\cap \pi\cap\P|=qC_{m-3}+1$, then an easy count shows that all hyperplanes through $\nu$, different from $\pi$, have $C_{m-1}$ points of $\P$. We find that $|H\cap \P'|$ is either $C_{m-1}$, $C_{m-1}-(qC_{m-3}+1)+(qB_{m-3}+1)$, or  $C_{m-1}-(qC_{m-3}+1)+(qA_{m-3}+1)$. Since $qB_{m-3}-qC_{m-3}=B_{m-1}-C_{m-1}$ and $qA_{m-3}-qC_{m-3}=A_{m-1}-C_{m-1}$ we indeed have that $|H\cap \P'|\in \{A_{m-1},B_{m-1},C_{m-1}\}$. 

Finally, suppose that $|H\cap \pi\cap\P|=qB_{m-3}+1$. Note that, unlike in the case of $\Q(2n,q)$, it does not immediately follow that all hyperplanes through $H\cap \pi$ have $B_{m-1}$ points. But we will show that this property still holds.

From the above reasoning we deduce that every singular hyperplane (i.e. meeting $\P$ in $B_{m-1}$ points), different from $\pi$, needs to intersect $\pi\cap \P$ in exactly $qB_{m-3}+1$ points. 
Furthermore, since we assume that $|\P\cap\pi|=q|\Q(2n-2,q|+1$ and that every hyperplane of $\pi$ meets the cone $\P\cap \pi$ in $qA_{m-3}+1,qB_{m-3}+1=|\Q(2n-2,q)|$, or $qC_{m-1}+1$ points of $\P$. It follows that the number of hyperplanes of each of these three types meeting $\P\cap\pi$ is a constant, independent of the choice of $\P$. Hence, the number of hyperplanes of each type is the same as the number of hyperplanes intersecting a cone with base a $\Q(2n-2,q)$. By Lemma \ref{lem:classical2}, we know that in that case, all hyperplanes through a hyperplane of $\pi$ with $qB_{m-3}+1$ points are singular, while those different from $\pi$ through a hyperplane with $qA_{m-3}+1$ have $A_{m-1}$ points and those through a hyperplane with $qC_{m-3}+1$ have $C_{m-1}$ points. Since we have that all hyperplanes through a hyperplane of $\pi$ with $qA_{m-3}+1$ of $\P$ have $A_{m-1}$ points of $\P$ and those through a hyperplane with $qC_{m-3}+1$ of $\P$ have $C_{m-1}$ points of $\P$, and all hyperplanes with $B_{m-1}$ meet $\pi$ in a hyperplane of $\pi$ with $qB_{m-3}+1$ points, we conclude that all hyperplanes through a hyperplane of $\pi$ with $qB_{m-3}+1$ points of $\P$ have $B_{m-1}$ points.
It then follows, as in the other cases, that $|H\cap \P'|$ is either $B_{m-1}$, $B_{m-1}-(qB_{m-3}+1)+(qA_{m-3}+1)$, or  $B_{m-1}-(qB_{m-3}+1)+(qC_{m-3}+1)$. Since $qA_{m-3}-qB_{m-3}=A_{m-1}-B_{m-1}$ and $qC_{m-3}-qB_{m-3}=C_{m-1}-B_{m-1}$ we indeed have that $|H\cap \P'|\in \{A_{m-1},B_{m-1},C_{m-1}\}$. 
\end{proof}

\begin{thebibliography}{99}

\bibitem{fining} J. Bamberg, A. Betten, Ph. Cara, J. De Beule, M. Lavrauw, and M. Neunh\"offer. Finite Incidence Geometry. FinInG – a {\em GAP package}, version 1.4.1, 2019.

\bibitem{BB} R.C. Bose and R.C. Burton. A characterization of flat spaces in a finite geometry and the uniqueness of the Hamming and the MacDonald codes. {\em J. Combin. Theory} {\bf 1} (1966), 96--104.

\bibitem{BHJS} S. Barwick, A. Hui, W-A. Jackson, and J. Schillewaert. Characterising hyperbolic solids of $Q(4,q)$, $q$ even. {\em Des. Codes Cryptogr.}, {\bf 88 (1)} (2020), 33--39.

\bibitem{QD} F. De Clerck and N. Durante. {\em Constructions and Characterizations of Classical Sets in $\PG(n,q)$}, In: J. De Beule and L. Storme (eds.), Current research topics on Galois geometry, pages 61–84. Nova Academic Publishers, 2011.
\bibitem{QQ}
F. De Clerck, N. Hamilton, C. O'Keefe and T. Penttila.
Quasi-quadrics and related structures. {\em Aust. J. Combin} {\bf 22} (2000), 151--166.

\bibitem{SDW-JS2}
{S.~De Winter and  J.~Schillewaert.}
A characterization of finite polar spaces by intersection numbers. {\em Combinatorica}  {\bf 30} (2010), 25--45. 

\bibitem{GGG}
J.W.P. Hirschfeld, J.A. Thas,
{\em General Galois Geometries}. 2nd edition. Springer-Verlag, London, 2016.

\end{thebibliography}
\end{document}